\documentclass[a4paper]{amsart}

\usepackage{amssymb,ifthen,amsthm}

\usepackage{ifthen}

\newcounter{mylisti} \newcounter{mylistii}
\newcounter{nest}
\newcommand{\defaultlabel}{}

\newenvironment{mylist}[1]{%
  \addtocounter{nest}{1}
  \ifthenelse{\value{nest}=1}{%
    \renewcommand{\defaultlabel}{(\roman{mylisti})\hfill}}{%
    \renewcommand{\defaultlabel}{(\alph{mylistii})\hfill}}
  \begin{list}{\defaultlabel}{%
      \ifthenelse{\value{nest}=1}{\usecounter{mylisti}}{%
        \usecounter{mylistii}}
      
      \addtolength{\topsep}{1ex}
      \addtolength{\itemsep}{0.5ex}
      \settowidth{\labelwidth}{#1}
      \setlength{\leftmargin}{\labelwidth}
      \addtolength{\leftmargin}{\labelsep}}}{\addtocounter{nest}{-1}
\end{list}}

\newcommand{\bn}{\ensuremath{\mathbb N}}

\newcommand{\br}{\ensuremath{\mathbb R}}

\newcommand{\cA}{\ensuremath{\mathcal A}}

\newcommand{\wt}{\ensuremath{\tilde{w}}}

\newcommand{\yt}{\ensuremath{\tilde{y}}}
\newcommand{\Et}{\ensuremath{\tilde{E}}}

\newcommand{\Gt}{\ensuremath{\tilde{G}}}
\newcommand{\Qt}{\ensuremath{\tilde{Q}}}
\newcommand{\Yt}{\ensuremath{\tilde{Y}}}
\newcommand{\Xt}{\ensuremath{\tilde{X}}}

\newcommand{\bbar}{\ensuremath{\bar{b}}}
\newcommand{\xb}{\ensuremath{\bar{x}}}

\newcommand{\Cb}{\ensuremath{\overline{C}}}

\newcommand{\Kb}{\ensuremath{\overline{K}}}

\newcommand{\deltab}{\ensuremath{\bar{\delta}}}
\newcommand{\etab}{\ensuremath{\bar{\eta}}}
\newcommand{\veb}{\ensuremath{\bar{\ve}}}
\newcommand{\Us}{\ensuremath{U^{(*)}}}
\newcommand{\Vs}{\ensuremath{V^{(*)}}}
\newcommand{\Ws}{\ensuremath{W^{(*)}}}
\newcommand{\Zs}{\ensuremath{Z^{(*)}}}

\newcommand{\cof}{\ensuremath{\mathrm{cof}}}

\newcommand{\ran}{\ensuremath{\mathrm{ran}}}

\newcommand{\supp}{\ensuremath{\mathrm{supp}}}
\newcommand{\sz}{\ensuremath{\mathrm{Sz}}} 

\newcommand{\vegtelen}[1]{{#1}^{(\omega)}}

\newcommand{\kcdot}{\!\cdot\!}
\newcommand{\keq}{\!=\!}

\newcommand{\kin}{\!\in\!}
\newcommand{\kge}{\!>\!}
\newcommand{\kgeq}{\!\geq\!}

\newcommand{\kle}{\!<\!}
\newcommand{\kleq}{\!\leq\!}

\newcommand{\kminus}{\!-\!}
\newcommand{\kneq}{\!\neq\!}

\newcommand{\knotin}{\!\notin\!}
\newcommand{\kplus}{\!+\!}

\newcommand{\ksetminus}{\!\setminus\!}

\newcommand{\ksubset}{\!\subset\!}

\newcommand{\ksupset}{\!\supset\!}

\newcommand{\ktimes}{\!\times\!}

\newcommand{\abs}[1]{\lvert #1\rvert}

\newcommand{\norm}[1]{\lVert #1\rVert}
\newcommand{\bignorm}[1]{\big\lVert #1\big\rVert}
\newcommand{\Bignorm}[1]{\Big\lVert #1\Big\rVert}
\newcommand{\biggnorm}[1]{\bigg\lVert #1\bigg\rVert}
\newcommand{\tnorm}[1]{\lvert\mspace{-1mu}\lvert\mspace{-1mu}\lvert
  #1\rvert\mspace{-1mu}\rvert\mspace{-1mu}\rvert}

\newcommand{\Bigtnorm}[1]{\Big\lvert\mspace{-1mu}\Big\lvert\mspace{-1mu}%
  \Big\lvert #1\Big\rvert\mspace{-1mu}\Big\rvert\mspace{-1mu}%
  \Big\rvert}

\newcommand{\coo}{\mathrm{c}_{00}}

\newcommand{\V}{\forall \,}
\newcommand{\E}{\exists \,}
\newcommand{\ve}{\varepsilon}

\newcommand{\ds}{\displaystyle}

\newcommand{\phtm}[1]{\text{\makebox[0pt]{\phantom{$#1$}}}}

\newcommand{\ie}{\textit{i.e.,}\ }

\newtheorem{thm}{Theorem}
\newtheorem{lem}[thm]{Lemma}
\newtheorem{prop}[thm]{Proposition}
\newtheorem{cor}[thm]{Corollary}

\theoremstyle{definition}

\newtheorem*{defn}{Definition}

\theoremstyle{remark}

\newtheorem*{rem}{Remark}

\newcommand{\odd}{\ensuremath{T_\infty^{\text{even}}}}

\author{E. Odell}
\address{Department of Mathematics \\ The University of Texas\\1
  University Station C1200\\ Austin, TX 78712\\ USA}
\email{odell@math.utexas.edu}

\author{Th. Schlumprecht}
\address{Department of Mathematics, Texas A\&M University\\ College
  Station, TX 78712, USA}
\email{thomas.schlumprecht@math.tamu.edu}
\author{A. Zs\'ak}
\address{School of Mathematical Sciences, University of Nottingham,
  University Park, Nottingham NG7 2RD, United Kingdom}
\email{andras.zsak@maths.nottingham.ac.uk}

\thanks{Research of the first two authors was supported by the
  National Science Foundation}

\title[A new infinite game]{A new infinite game in Banach spaces with
  applications}

\date{\today}


\begin{document}

\maketitle

\begin{center}
  \emph{Dedicated to Nigel J.~Kalton on the occassion of his
    $60^\mathrm{th}$ birthday.}
\end{center}

\section{Introduction}

Let $X$ be a separable infinite-dimensional Banach space, and let
$\cA$ be a set of normalized sequences in $X$. We can consider
a two-player game in $X$ each move of which consists of player S
(subspace chooser) selecting some element $Y$ from the set $\cof(X)$
of finite-codimensional subspaces of $X$, and P (point chooser)
responding by selecting a vector $y$ from the unit sphere $S_Y$ of
$Y$. The game, which we shall refer to as the $\cA$-game, consists of
an infinite sequence of such moves generating a sequence $X_1, x_1,
X_2, x_2,\dots$, where $X_i\kin\cof(X)$ and $x_i\kin S_{X_i}$ for all
$i\kin\bn$. S wins the $\cA$-game if $(x_i)_{i=1}^\infty\kin\cA$.

Of course this game, which has its roots in the game described by
W.~T.~Gowers~\cite{G} and in the notion of asymptotic
structure~\cite{MMT}, has certain limitations. Unlike the theory of
asymptotic structure (where, for each $n\kin\bn$, a game is considered
that consists of $n$ moves, where each move is the same as above),
there is generally no unique smallest class $\cA$ (depending on $X$)
for which S has a winning strategy. However, one can hypothesize
certain specific classes $\cA$ for which S has a winning strategy for
a given $X$ and deduce certain structural consequences. For example,
if for some $K\kge 0$ we let $\cA$ be the class of sequences
$K$-equivalent to the unit vector basis of $\ell_p$ ($1\kle
p\kle\infty$), then any reflexive space $X$ in which S has a winning
strategy for the $\cA$-game in $X$, embeds into an $\ell_p$-sum of
finite-dimensional spaces~\cite{OS1}. In fact, it was the problem of
classifying subspaces of $\ell_p$-sums of finite-dimensional spaces
that motivated the study of this game.

The general theme here is to take a coordinate-free property of a
space $X$, recast it in terms of S having a winning strategy in the
$\cA$-game for a suitable class $\cA$, and then to show that $X$
embeds into a space with an FDD (finite-dimensional decomposition)
which has the ``coordinatized'' version of the property we started
with. In addition to the $\ell_p$ result in~\cite{OS1} cited above this
general theme was followed in~\cite{OS2} and~\cite{OSZ1}. In~\cite{OS2}
reflexive spaces $X$ were studied for which S has a winning strategy
for both games corresponding to the classes $\cA_p$ of normalized
basic sequences with an $\ell_p$-lower estimate and $\cA^q$ of
normalized basic sequences with an $\ell_q$-upper estimate ($1\kle
q\kleq p\kle\infty$). The end result was that $X$ embeds into a
reflexive space with an FDD such that every block sequence satisfies
$\ell_p$-lower and $\ell_q$-upper estimates. A consequence of this is
that one can construct a separable, reflexive space universal for the
class of separable, uniformly convex spaces or, more generally, for
the class $C_\omega\keq \{X:\,X\text{ is separable, reflexive },
\sz(X)\kleq\omega,\ \sz(X^*)\kleq\omega\}$, where $\sz(Y)$ denotes the
Szlenk index of a separable Banach space $Y$. Recently an alternative
proof of the universal result was given~\cite{DF} using powerful
set-theoretical notions (although, the FDD structural results cannot
be obtained in this way). We should also note that a set-theoretical
study of $\cA$-games was given by C.~Rosendal~\cite{R}.

The motivation behind this paper arose from a problem posed to us by
A. Pe\l czy\'nski. Given $\alpha\kle\omega_1$, does there exist a
separable, reflexive space universal for the class $C_\alpha$ (defined
as above with $\omega$ replaced by $\alpha$). Thus far the authors
of~\cite{DF} have been unable to extend their techniques to this
problem. In researching Pe\l czy\'nski's problem we discovered that it
was necessary to consider a new game and solve the corresponding
embedding problem in this context.

The game is played as follows. In each move of the game S (subspace
chooser) selects $k\kin\bn$ and $Y\kin\cof(X)$, and then P (point
chooser) responds by choosing $y\kin S_Y$. The game then consists of
an infinite sequence of moves generating a sequence
$k_1,X_1,x_1,k_2,X_2,x_2,\dots$, where $k_i\kin\bn,\ X_i\kin\cof(X)$
and $x_i\kin S_{X_i}$ for all $i\kin\bn$. Given a normalized,
$1$-unconditional sequence $(v_i)$, S is declared winner of this game
if $(v_{k_i})$ is dominated by $(x_i)$. We also
consider, for given normalized $1$-unconditional sequence $(u_i)$, the
version where S wins the game if $(x_i)$ is dominated
by $(u_{k_i})$. In the case where $(v_i)$ and $(u_i)$ are the unit
vector bases of $\ell_p$ and $\ell_q$, respectively, this conforms to
the games considered in~\cite{OS2}, but Pe\l czy\'nski's problem
requires us to consider sequences $(v_i)$ and $(u_i)$ that are not
subsymmetric. (In~\cite{OSZ1} in order to solve the problem of
embedding an asymptotic $\ell_p$ space into one with an asymptotic
$\ell_p$ FDD it was necessary to extend the results of~\cite{OS2}
concerning $\ell_p$-lower and $\ell_q$-upper estimates to more general
$(v_i)$-lower and $(u_i)$-upper estimates, but the game played did not
change.)

The main results of this paper are given in
Section~\ref{section:embedding}, Theorems~\ref{thm:embedding}
and~\ref{thm:lower-and-upper-embedding}, and
Corollary~\ref{cor:subspace-quotient}. In brief these theorems say the
following. Suppose we are given normalized, $1$-unconditional bases
$(v_i)$ and $(u_i)$ with certain properties, and a reflexive space
$X$. Assume that S wins the subsequential $(v_i)$-lower and the
subsequential $(u_i)$-upper games described above. Then $X$ embeds
into a space $Z$ with an FDD  such that every block sequence satisfies
subsequential $(v_i)$-lower and $(u_i)$-upper estimates. (Precise
definitions of these estimates will be given below.)

One application of these theorems is a new proof of the results
of~\cite{OS2}. The application to the Pe\l czy\'nski problem will appear
in~\cite{OSZ2}, where further machinery is necessary to exploit the
results obtained here.

In Section~\ref{section:universal} we derive some universal space
consequences of our embedding theorems. Section~\ref{section:prelim}
introduces our terminology, in particular we give precise definitions
of various lower and upper norm
estimates. Section~\ref{section:prelim} also contains some
straightforward duality results concerning such norm estimates, and a
combinatorial result (Proposition~\ref{prop:tree-est-sb-est}) that is
key to embedding spaces satisfying the coordinate-free version of a
certain property into a space with an FDD satisfying the
``coordinatized'' version of the same property.

In Section~\ref{section:zv} we define the space $Z^V(E)$, where $Z$ is
a Banach space with an FDD $E\keq(E_i)$ and $V$ is the closed linear
span of a normalized, $1$-unconditional sequence $(v_i)$. We
develop the properties of $Z^V(E)$, in particular proving that, under
appropriate hypotheses,  $Z^V(E)$ is a reflexive space admitting
subsequential $V$-lower estimates.

\section{Definitions and preliminary results}
\label{section:prelim}

We begin with fixing some terminology. Let $Z$ be a Banach space with
an FDD $E\keq(E_n)$.  For $n\kin\bn$ we denote by $P^E_n$ the $n$-th
\emph{coordinate projection}, \ie $P^E_n\colon Z\to E_n$ is the map
defined by $\sum_i z_i\mapsto z_n$, where $z_i\kin E_i$ for all
$i\kin\bn$. For a finite set $A\ksubset\bn$ we put
$P^E_A\keq\sum_{n\in A} P^E_n$. The \emph{projection constant $K(E,Z)$
  of $(E_n)$ (in $Z$) }is defined by
\[
K=K(E,Z)=\sup_{m<n}\norm{P^E_{[m,n]}}\ ,
\]
where $[m,n]$ denotes the interval $\{m,m\kplus 1,\dots, n\}$ in
$\bn$. Recall that $K$ is always finite and, as in the case of bases,
we say that \emph{$(E_n)$ is bimonotone (in $Z$) }if $K\keq 1$. By
passing to the equivalent norm
\[
\tnorm{\cdot}\colon Z\to\br\ ,\qquad z\mapsto\sup_{m<n}
\norm{P^E_{[m,n]}(z)}\ ,
\]
we can always renorm $Z$ so that $K\keq 1$.

For a sequence $(E_i)$ of finite-dimensional spaces
we define the vector space
\[
\coo(\oplus_{i=1}^\infty E_i)=\big\{(z_i):\,z_i\kin E_i \text{ for all
  $i\kin\bn$, and $\{i\kin\bn:\,z_i\kneq 0\}$ is finite}\big\}\ ,
\]
which is dense in each Banach space for which $(E_i)$ is an FDD. For a
set $A\ksubset\bn$ we denote  by $\bigoplus_{i\in A} E_i$ the linear
subspace of $\coo(\oplus E_i)$ generated by the elements of
$\bigcup_{i\in A}E_i$. As usual we denote the vector space of
sequences in $\br$ which are eventually zero by $\coo$. We sometimes
will consider for the same
sequence $(E_i)$ of finite-dimensional spaces different norms on
$\coo(\oplus E_i)$. In order to avoid confusion we will therefore
often index the norm by the Banach space whose norm we are using, \ie
$\norm{\cdot}_Z$ denotes the norm of the Banach space $Z$.

If $Z$ has an FDD $(E_i)$, the vector space
$\coo(\oplus_{i=1}^\infty E^*_i)$, where $E^*_i$ is the dual space of
$E_i$ for each $i\kin\bn$, can be identified in a natural way with a
$w^*$-dense subspace of $Z^*$. Note however
that the embedding $E^*_i\hookrightarrow Z^*$ is, in general, not
isometric unless $K\keq 1$. We will always consider $E^*_i$ with the
norm it inherits from $Z^*$ instead of the norm it has as the dual
space of $E_i$. We denote the norm closure of
$\coo(\oplus_{i=1}^\infty E^*_i)$ in $Z^*$ by $\Zs$. Note that $\Zs$
is $w^*$-dense in $Z^*$, the unit ball $B_{\Zs}$ norms $Z$, and
$(E_i^*)$ is an FDD of $\Zs$ having a projection constant not
exceeding $K(E,Z)$. If $K(E,Z)\keq1$, then $B_{\Zs}$ is $1$-norming
for $Z$ and $Z^{(*)(*)}\keq Z$.

For $z\kin \coo(\oplus E_i)$ we define \emph{the support $\supp_E(z)$
  of $z$ with respect to $(E_i)$ }by
\[
\supp_E(z)=\{i\kin\bn:\,P^E_i(z)\kneq 0\}\ ,
\]
and we define the \emph{range $\ran_E(z)$ of $Z$ with respect to
$(E_i)$ }to be the smallest interval in $\bn$ containing
$\supp_E(z)$. A sequence $(z_i)$ (finite or infinite) of non-zero
vectors in $\coo(\oplus E_i)$ is called \emph{a block sequence of
  $(E_i)$ }if
\[
\max\supp_E(z_n)<  \min\supp_E(z_{n+1})\qquad\text{whenever $n\kin\bn$
  (or $n\kle\,$length$(z_i)$)}\ .
\]
A block sequence $(z_i)$ of $(E_i)$ is called \emph{normalized (in
  $Z$) }if $\norm{z_i}_Z\keq 1$ for all $i\kin\bn$.

Let $\deltab\keq(\delta_i)\ksubset (0,1)$ with $\delta_i\downarrow
0$. A (finite or infinite) sequence $(z_i)$ in $S_Z$ is called
\emph{a $\deltab$-skipped block sequence of $(E_i)$ }if there exists a
sequence $1\kleq k_0\kle k_1\kle k_2\kle\dots$ in $\bn$ such that
\[
\norm{z_n-P^E_{(k_{n-1},k_n)}(z_n)}<\delta_n\qquad\text{for all
  $n\kin\bn$ (or $n\kleq\,$length$(z_i)$)}\ .
\]
\begin{rem}
  A sequence $(F_i)$ of finite-dimensional spaces is  called \emph{a
    blocking of $(E_i)$ }if for some sequence $m_1\kle m_2\kle\dots$
  in $\bn$ we have $F_n\keq\bigoplus_{j=m_{n-1}+1}^{m_n}E_j$ for all
  $n\kin\bn$ ($m_0\keq 0$).  
  If $(F_i)$ is a blocking of $(E_i)$, and if $(x_i)$ is a
  $\deltab$-skipped block sequence of $(F_i)$, then $(x_i)$ is not
  necessarily a $\deltab$-skipped block sequence of $(E_i)$ (since  in
  the definition of skipped block sequence we skip exactly one
  coordinate). Nevertheless it is clear that $(x_i)$ is a
  $2K\deltab$-skipped block sequence of $(E_i)$, where $K$ is the
  projection constant of $(E_i)$ in $Z$.
\end{rem}

\begin{defn}
  Given two sequences $(e_i)$ and $(f_i)$ in some Banach spaces, and
  given a constant $C\kge 0$, we say that \emph{$(f_i)$ $C$-dominates
    $(e_i)$, }or that \emph{$(e_i)$ is $C$-dominated by $(f_i)$}, if
  \[
  \Bignorm{\sum a_ie_i}\leq C  \Bignorm{\sum a_if_i}\qquad\text{for
    all $(a_i)\kin\coo$}\ .
  \]  
  We say that \emph{$(f_i)$ dominates $(e_i)$, }or that \emph{$(e_i)$
    is dominated by $(f_i)$, }if there exists a constant $C\kge 0$ such
  that $(f_i)$ $C$-dominates $(e_i)$.
\end{defn}

We shall now introduce certain lower and upper norm estimates for
FDD's.

\begin{defn}
  Let $Z$ be a Banach space with an FDD $(E_n)$, let $V$ be a Banach
  space with a normalized, $1$-unconditional basis $(v_i)$ and let
  $1\kleq C\kle\infty$.
  
  We say that $(E_n)$ \emph{satisfies subsequential $C$-$V$-lower
  estimates (in $Z$) }if every normalized block sequence $(z_i)$ of
  $(E_n)$ in $Z$ $C$-dominates $(v_{m_i})$, where
  $m_i\keq\min\supp_E(z_i)$ for all $i\kin\bn$, and $(E_n)$
  \emph{satisfies subsequential $C$-$V$-upper estimates (in $Z$) }if
  every normalized block sequence $(z_i)$ of $(E_n)$ in $Z$ is
  $C$-dominated by $(v_{m_i})$, where $m_i\keq\min\supp_E(z_i)$ for
  all $i\kin\bn$.

  If $U$ is another space with a normalized and 1-unconditional basis
  $(u_i)$, we say that $(E_n)$ \emph{satisfies subsequential
  $C$-$(V,U)$ estimates (in $Z$) }if it satisfies subsequential
  $C$-$V$-lower and $C$-$U$-upper estimates in~$Z$.

  We say that $(E_n)$ satisfies \emph{subsequential $V$-lower,
  $U$-upper }or \emph{$(V,U)$ estimates (in $Z$) }if for some $C\kgeq
  1$ it satisfies subsequential $C$-$V$-lower, $C$-$U$-upper or
  $C$-$(V,U)$ estimates in $Z$, respectively.
\end{defn}

\begin{rem}
  Assume that $(E_n)$ satisfies subsequential $C$-$V$-lower estimates
  in $Z$ and that $(z_i)$ is a  normalized block sequence of
  $(E_n)$. If $\max\supp_E(z_{i-1})\kle m_i\kleq \min\supp_E(z_i)$ for
  all $i\kin\bn$ (where $\max\supp_E(z_0)\keq 0$), then $(z_i)$
  $C$-dominates $(v_{m_i})$.

  Another easy fact is that if every normalized block sequence $(z_i)$
  of $(E_n)$ in $Z$ dominates $(v_{m_i})$, where
  $m_i\keq\min\supp_E(z_i)$ for all $i\kin\bn$, then $(E_n)$ satisfies
  subsequential $V$-lower estimates in $Z$.

  Analogous statements hold for upper estimates.
\end{rem}

We shall need a coordinate-free version of subsequential lower and
upper estimates. One way of defining this is reminiscent of the notion
of asymptotic structure. Let $V$ be a Banach space with a normalized
and $1$-unconditional basis $(v_i)$, and let $C\kin[1,\infty)$. Assume
that we are given an infinite-dimensional Banach space $X$. We say
that \emph{$X$ satisfies subsequential $C$-$V$-lower estimates
}(respectively, \emph{subsequential $C$-$V$-upper estimates}) if
\[
\begin{array}{ccc}
  \E k_1\kin\bn & \E X_1\kin\text{cof}(X) & \V x_1\kin S_{X_1} \\
  \E k_2\kin\bn & \E X_2\kin\text{cof}(X) & \V x_2\kin S_{X_2} \\
  \E k_3\kin\bn & \E X_3\kin\text{cof}(X) & \V x_3\kin S_{X_3} \\
  \vdots & \vdots & \vdots
\end{array}
\]
such that $k_1\kle k_2\kle\dots$ and
$(v_{k_i})$ is $C$-dominated by (respectively, $C$-dominates)
$(x_i)$. If $U$ is another Banach space with a normalized,
$1$-unconditional basis $(u_i)$, then we say that \emph{$X$ satisfies
  subsequential $C$-$(V,U)$ estimates }if it satisfies subsequential
$C$-$V$-lower and $C$-$U$-upper estimates. Finally, we say that
\emph{the Banach space $X$ satisfies subsequential $V$-lower,
  $U$-upper or $(V,U)$ estimates }if for some constant $C$ it
satisfies subsequential $C$-$V$-lower, $C$-$U$-upper or $C$-$(V,U)$
estimates, respectively.

The above definitions are given more formally in the language of
games. Let us recall from the Introduction that in our games each move
consists of S (subspace chooser) selecting $k\kin\bn$ and
$Y\kin\cof(X)$, and then P (point chooser) responding by choosing
$y\kin S_Y$. The game then consists of an infinite sequence of moves
generating a sequence $k_1,X_1,x_1,k_2,X_2,x_2,\dots$, where
$k_i\kin\bn,\ X_i\kin\cof(X)$ and $x_i\kin S_{X_i}$ for all
$i\kin\bn$. Player S wins the game if $(v_{k_i})$ is $C$-dominated by
$(x_i)$, otherwise P is declared the winner. The space $X$ then
satisfies subsequential $C$-$V$-lower estimates if and only if S has a
winning strategy, \ie there is a function $\phi$ such that given
sequences $(k_i)$ in $\bn$, $(X_i)$ in cof$(X)$, and $(x_i)$ in $X$
such that $x_i\kin S_{X_i}$ \emph{and} $(k_i,X_i)\keq
\phi(x_1,x_2,\dots,x_{i-1})$ for all $i\kin\bn$, then $k_1\kle
k_2\kle\dots$ and $(v_{k_i})$ is $C$-dominated by $(x_i)$. The notions
of subsequential $C$-$V$-upper estimates, subsequential $C$-$(V,U)$
estimates, etc.~can be formalized in a similar way.

Yet another way of expressing subsequential lower and upper estimates
in a coordinate-free way uses infinite, countably branching trees
(see Proposition~\ref{prop:game-tree-def} below). This is not
surprising since a winning strategy in the game described above
corresponds naturally to such a tree. We define  for $\ell\kin\bn$
\begin{align*}
  T_\ell&=\big\{(n_1,n_2,\dots,n_\ell):\,n_1\kle n_2\kle\dots\kle
  n_\ell \text{ are in }\bn\big\}
  \intertext{and}
  T_\infty&=\bigcup_{\ell=1}^\infty T_\ell\ ,\qquad
  \odd =\bigcup_{\ell=1}^\infty T_{2\ell}\ .
\end{align*}
If $\alpha\keq(m_1,\dots,m_\ell)\kin T_\ell$, we call $\ell$ the
\emph{length of $\alpha$} and denote it by $\abs{\alpha}$, and
$\beta\keq(n_1,\dots,n_k)\kin T_\infty$ is called an \emph{extension
  of $\alpha$}, or $\alpha$ is called \emph{a restriction of }$\beta$,
if $\ell\kle k$ and $m_i\keq n_i$ for $i\keq 1,\dots,\ell$. We then
write $\alpha\kle \beta$ and with this order both $T_\infty$ and
$\odd$ are trees.

An \emph{even tree }in a Banach space $X$ is a family
$(x_\alpha)_{\alpha\in\odd}$ in $X$. Sequences of
the form $\big(x_{(\alpha,n)}\big)_{n>n_{2\ell-1}}$, where
$\ell\kin\bn$ and $\alpha\keq (n_1, n_2,\dots, n_{2\ell-1})\kin
T_\infty$, are called \emph{nodes }of the tree. For a sequence
$n_1\kle n_2\kle\dots$ of positive integers the sequence
$\big(x_{(n_1,n_2,\dots,n_{2\ell})}\big)_{\ell=1}^\infty$ is called a
\emph{branch }of the tree.

If $(x_\alpha)_{\alpha\in\odd}$ is an even tree in a Banach space $X$
and if $T'\ksubset\odd$  is closed under taking restrictions so that
for each $\alpha\kin T'\cup\{\emptyset\}$ and for each $m\kin\bn$ the
set $\{n\kin\bn:\,(\alpha,m,n)\kin T'\}$ is either empty or has
infinite size, and moreover the latter occurs for infinitely many
values of $m$, then we call $(x_\alpha)_{\alpha\in T'}$ a \emph{full
  subtree of }$(x_\alpha)_{\alpha\in\odd}$. Note that
$(x_\alpha)_{\alpha\in T'}$ could then be relabeled to a family
indexed by $\odd$, and note that the branches of
$(x_\alpha)_{\alpha\in T'}$ are branches of
$(x_\alpha)_{\alpha\in\odd}$ and that the nodes of
$(x_\alpha)_{\alpha\in T'}$ are subsequences of certain nodes of
$(x_\alpha)_{\alpha\in\odd}$.

An even tree $(x_\alpha)_{\alpha\in\odd}$ in a Banach space $X$ is
called \emph{normalized }if $\norm{x_{\alpha}}\keq 1$ for all
$\alpha\kin\odd$, and is called \emph{weakly null }if every node is a
weakly null sequence. If $X$ has an FDD $(E_n)$, then
$(x_\alpha)_{\alpha\in\odd}$ is called a \emph{block even tree of
  $(E_n)$ }if every node is a block sequence of $(E_n)$.

\begin{defn}
  \label{def:tree-est}
  Let $V$ be a Banach space with a normalized and $1$-unconditional
  basis $(v_i)$, and let $C\kin[1,\infty)$. Assume that we are given
  an infinite-dimensional Banach space $X$. We say that \emph{$X$
  satisfies subsequential $C$-$V$-lower tree estimates }if every
  normalized, weakly null even tree $(x_\alpha)_{\alpha\in\odd}$ in
  $X$ has a branch $\big(x_{(n_1,n_2,\dots,n_{2i})}\big)$ which
  $C$-dominates $(v_{n_{2i-1}})$.

  We say that $X$ \emph{satisfies subsequential $C$-$V$-upper tree
  estimates }if every normalized, weakly null even tree
  $(x_\alpha)_{\alpha\in\odd}$ in $X$ has a branch
  $\big(x_{(n_1,n_2,\dots,n_{2i})}\big)$ which is $C$-dominated by
  $(v_{n_{2i-1}})$.

  If $U$ is a second space with a $1$-unconditional and normalized
  basis $(u_i)$, we say that $X$ \emph{satisfies subsequential
  $C$-$(V,U)$ tree estimates }if it satisfies subsequential
  $C$-$V$-lower and $C$-$U$-upper tree estimates.

  We say that $X$ \emph{satisfies subsequential $V$-lower, $U$-upper
  }or \emph{$(V,U)$ tree estimates }if for some $1\kleq C\kle\infty$
  $X$ satisfies subsequential $C$-$V$-lower, $C$-$U$-upper or
  $C$-$(V,U)$ tree estimates, respectively.
\end{defn}

\begin{rem}
  As in the FDD case, we do not need to fix a constant $C$ in the
  above definitions: if every normalized, weakly null even tree
  $(x_\alpha)_{\alpha\in\odd}$ in $X$ has a branch
  $\big(x_{(n_1,n_2,\dots,n_{2i})}\big)$ which dominates
  $(v_{n_{2i-1}})$, then there exists a constant $C\kgeq 1$ such that
  $X$ satisfies subsequential $C$-$V$-lower tree estimates. The
  analogous statement for upper estimates also
  holds. (See~\cite[Proposition~1.2]{OSZ1}.)
\end{rem}

Proposition~\ref{prop:game-tree-def} below shows that, under some mild
hypotheses, the two coordinate-free versions of lower and upper
estimates given above are essentially the same. Before stating this
result we need a certain property of basic sequences defined
in~\cite{BHO}.

\begin{defn}
  Let $V$ be a Banach space with a normalized, $1$-unconditional basis
  $(v_i)$ and let $1\kleq C\kle\infty$.

  We say that $(v_i)$ is \emph{$C$-right-dominant }(respectively,
  \emph{$C$-left-dominant}) if for all sequences $m_1\kle m_2\kle
  \dots$ and $n_1\kle n_2\kle\dots$ of positive integers with
  $m_i\kleq n_i$ for all $i\kin\bn$ we have that $(v_{m_i})$ is
  $C$-dominated by (respectively, $C$-dominates) $(v_{n_i})$. We say
  that $(v_i)$ is \emph{right-dominant }or \emph{left-dominant }if for
  some $C\kgeq 1$ it is $C$-right-dominant or $C$-left-dominant,
  respectively.
\end{defn}

\begin{rem}
  For $(v_i)$ to be right-dominant (respectively, left-dominant) it is
  enough to have the property that $(v_{m_i})$ is dominated by
  (respectively, dominates) $(v_{n_i})$ for all sequences $m_1\kle
  m_2\kle \dots$ and $n_1\kle n_2\kle\dots$ of positive integers with
  $m_i\kleq n_i$ for all $i\kin\bn$. Also, $(v_i)$ is
  $C$-right-dominant (respectively, $C$-left-dominant) if and only if
  the sequence $(v^*_i)$ of biorthogonal functionals in $\Vs$ is
  $C$-left-dominant (respectively, $C$-right-dominant).
\end{rem}

\begin{prop}
  \label{prop:game-tree-def}
  Let $V$ be a Banach space with a normalized and $1$-unconditional
  basis $(v_i)$ and let $C,D\kin [1,\infty)$. Let $X$ be an
  infinite-dimensional Banach space.

  \begin{mylist}{(a)}
  \item[(a)]
    Assume that $(v_i)$ is $D$-left-dominant. If $X$ satisfies
    subsequential $C$-$V$-lower estimates, then for all $\ve\kge 0$
    $X$ satisfies subsequential $(CD\kplus\ve)$-$V$-lower tree
    estimates.
  \item[(b)]
    Assume that $X^*$ is separable. If $X$ satisfies subsequential
    $C$-$V$-lower tree estimates, then it also satisfies subsequential
    $C$-$V$-lower estimates.
  \end{mylist}
\end{prop}

\begin{rem}
  Analogous results hold for upper estimates. For that in~(a) we need
  to assume that $(v_i)$ is $D$-right-dominant.
\end{rem}

\begin{proof}
  (a) Assume that for some $\ve\kge 0$ there is a normalized, weakly
  null even tree $(x_\alpha)_{\alpha\in\odd}$ in $X$ such that for any
  sequence $n_1\kle n_2\kle\dots$ of positive integers the branch
  $\big(x_{(n_1,n_2,\dots,n_{2i})}\big)$ does not
  $(CD\kplus\ve)$-dominate $(v_{n_{2i-1}})$. We show that in this case
  $X$ does not satisfy subsequential $C$-$V$-lower estimates by
  exhibiting a winning strategy for the point chooser P. Fix a
  sequence $(\delta_i)\ksubset (0,1)$ with $\Delta\keq\sum_i\delta_i$
  satisfying
  \[
  C<\frac{CD+\ve}D\cdot\left( 1+\Delta\frac{CD+\ve}D\right)^{-1}\ .
  \]
  Suppose the game starts with S picking $k_1\kin\bn$ and
  $X_1\kin\text{cof}(X)$. Since the nodes of
  $(x_\alpha)_{\alpha\in\odd}$ are weakly null, there exist
  $n_1,n_2\kin\bn$ such that $k_1\kleq n_1\kle n_2$ and
  $d\big(x_{(n_1,n_2)},X_1\big)\kle\delta_1$. P's response will be a
  $y_1\kin S_{X_1}$ with $\bignorm{x_{(n_1,n_2)}\kminus
  y_1}\kle\delta_1$. In the second move $S$ picks $k_2\kin\bn$ and
  $X_2\kin\text{cof}(X)$. Then there exist $n_3,n_4\kin\bn$ such that
  $n_2\kle n_3\kle n_4,\ k_2\kleq n_3$ and
  $d\big(x_{(n_1,n_2,n_3,n_4)},X_2\big)\kle\delta_2$. P's response
  will be some $y_2\kin S_{X_2}$ with
  $\bignorm{x_{(n_1,n_2,n_3,n_4)}\kminus y_2}\kle\delta_2$. In
  general, on the $j^{\text{th}}$ move of the game ($j\kgeq 2$), S
  picks $k_j\kin\bn,\ X_j\kin\text{cof}(X)$. Then one can find
  $n_{2j-1}, n_{2j}\kin\bn$ such that $n_{2j-2}\kle n_{2j-1}\kle
  n_{2j},\ k_j\kleq n_{2j-1}$ and
  $d\big(x_{(n_1,n_2,\dots,n_{2j})},X_j\big)\kle\delta_j$. P's
  $j^{\text{th}}$ move will be some $y_j\kin S_{X_j}$ such that
  $\bignorm{x_{(n_1,n_2,\dots,n_{2j})}\kminus y_j}\kle\delta_j$.

  Since the branch $\big(x_{(n_1,n_2,\dots,n_{2i})}\big)$ does not
  $(CD\kplus\ve)$-dominate $(v_{n_{2i-1}})$, there exists
  $(a_i)\kin\coo$ such that
  \[
  \Bignorm{\sum _ia_iv_{n_{2i-1}}} > (CD\kplus\ve) \Bignorm{\sum _ia_i
    x_{(n_1,n_2,\dots,n_{2i})}}\ .
  \]
  We may assume without loss of generality that $\bignorm{\sum a_i
  v_{k_i}}\keq 1$. Using the $D$-left-dominant property of $(v_i)$ and
  that $k_i\kleq n_{2i-1}$ for all $i\kin\bn$, together with the
  choice of $\Delta$, an easy computation now gives
  \[
  \Bignorm{\sum _ia_iv_{k_i}} > C \Bignorm{\sum _ia_i y_i}\ .
  \]
  Thus P wins the game.\\

  \noindent
  (b) Assume that $X$ does not satisfy subsequential $C$-$V$-lower
  estimates. This means that S does not have a winning strategy, which
  in turn implies that there is a winning strategy $\phi$ for the point
  chooser (this follows from the fact that closed games~\cite{GS} or,
  more generally, Borel games~\cite{M} are determined). Thus given
  sequences $(k_i)$ in $\bn$, $(X_i)$ in cof$(X)$,
  and $(x_i)$ in $X$ such that $k_1\kle k_2\kle\dots$ \emph{and}
  $x_n\keq \phi(k_1,X_1,k_2,X_2,\dots,k_n,X_n)$ for all $n\kin\bn$,
  then $x_i\kin S_{X_i}$ for all $i\kin\bn$ and $(v_{k_i})$ is not
  $C$-dominated by $(x_i)$. Fix a sequence $(X_i)$ of
  finite-codimensional subspaces of $X$ such that every bounded
  sequence $(x_i)$ with $x_i\kin X_i$ for all $i\kin\bn$ is weakly
  null. This exists by the assumption that $X$ has separable dual.

  We now construct a normalized, weakly null even tree in $X$ by
  recursion to show that $X$ does not satisfy subsequential
  $C$-$V$-lower tree estimates. For $\ell\kin\bn$ and
  $\alpha\keq(n_1,n_2,\dots,n_{2\ell})\kin T_{2\ell}$ we set
  $x_\alpha\keq\phi(n_1,X_{n_2},n_3,X_{n_4},%
  \dots,n_{2\ell-1},X_{n_{2\ell}})$. It is easy to verify that
  $(x_\alpha)_{\alpha\in \odd}$ is a normalized, weakly null even tree
  in $X$, and that for any sequence $n_1\kle n_2\kle\dots$ of positive
  integers the branch $\big(x_{(n_1,n_2,\dots,n_{2i})}\big)$ does not
  $C$-dominate $(v_{n_{2i-1}})$.
\end{proof}

If $V$ is a Banach space with a normalized, $1$-unconditional basis,
and if $N$ is an infinite subset of $\bn$, we write $V_N$ for the
closed linear span of $\{ v_i:\,i\kin N\}$. When we talk about
subsequential $V_N$-lower estimates, etc., it will be with respect to
the normalized, $1$-unconditional basis $(v_i)_{i\in N}$ of $V_N$.
We shall also write $\vegtelen{\bn}$ for the set of all infinite
subsets of $\bn$.

Note that if $V$ is a Banach space with a normalized,
$1$-unconditional, left-dominant basis $(v_i)$, then for any space $Z$
with an FDD $(E_i)$ satisfying subsequential $V$-lower estimates in
$Z$, the FDD $(E_i)$ will also satisfy subsequential $V_N$-lower
estimates for any $N\kin\vegtelen{\bn}$. Later on we shall need a
result that allows us to pass from subsequential $V_N$-lower estimates
for some $N\kin\vegtelen{\bn}$ to subsequential $V$-lower
estimates. Before stating this result we need a definition.

\begin{defn}
  Let $V$ be a Banach space with a normalized, $1$-unconditional basis
  $(v_i)$ and let $1\kleq C\kle\infty$.

  We say that $(v_i)$ is \emph{$C$-block-stable }if any two normalized
  block bases $(x_i)$ and $(y_i)$ with
  \[
  \max \big( \supp (x_i)\cup\supp(y_i)\big) < \min \big( \supp
  (x_{i+1})\cup\supp(y_{i+1})\big)\qquad\text{for all }i\kin\bn
  \]
  are $C$-equivalent. We say that $(v_i)$ is \emph{block-stable }if it
  is $C$-block-stable for some constant~$C$.
\end{defn}

\begin{rem}
  It is routine to check that $(v_i)$ is $C$-block-stable if and only
  if the sequence $(v^*_i)$ of biorthogonal functionals in $\Vs$ is
  $C$-block-stable.

  A block-stable basis is a special case of a
  \emph{block-norm-determined FDD }introduced by H.~P.~Rosenthal,
  who has initiated an exhaustive study of such FDDs~\cite{Ro}.
\end{rem}

\begin{lem}
  \label{lem:norm-est-for-subseq}
  Let $V$ and $U$ be Banach spaces with normalized, $1$-unconditional,
  block-stable bases $(v_i)$ and $(u_i)$, respectively, and assume
  that $(v_i)$ is dominated by $(u_i)$. Let $M\kin\vegtelen{\bn}$ and
  let $Z$ be a Banach space with an FDD $E\keq (E_i)$ satisfying
  subsequential $(V_M,U_M)$ estimates in $Z$. Then $W\keq
  Z\oplus_{\ell_\infty} V_{\bn\setminus M}$ has an FDD $F\keq (F_i)$
  satisfying subsequential $(V,U)$ estimates in $W$.
\end{lem}

\begin{proof}
  Choose constants $B,C,D\kin [1,\infty)$ such that $(v_i)$ and
  $(u_i)$ are $B$-block-stable, $(v_i)$ is $D$-dominated by $(u_i)$,
  and $(E_i)$ satisfies subsequential $C$-$(V_M,U_M)$ estimates in
  $Z$. For each $n\kin\bn$ define
  \begin{equation*}
    F_n=
    \begin{cases}
      E_i & \text{if $n\keq m_i$ for some $i\kin\bn$\ ,}\\
      \br\kcdot v_n & \text{if $n\knotin M$\ .}
    \end{cases}
  \end{equation*}
  Then $F\keq (F_n)$ is an FDD for $W$ with projection
  constant $K(F,W)\kleq K(E,Z)$. We now show that $(F_i)$ satisfies
  subsequential $\Cb$-$(V,U)$ estimates in $W$, where $\Cb\keq
  B\kcdot\max\{2C,D\}$.

  Let $(z_i)$ be a normalized block sequence of $(F_n)$ in $W$. For
  each $i\kin\bn$ let $k_i\keq \min\supp_F(z_i)$ and write
  \[
  z_i= z^{(1)}_i+z^{(2)}_i,\qquad\text{where }z^{(1)}_i\kin Z\ ,\
  z^{(2)}_i\kin V_{\bn\setminus M}\ .
  \]
  Fix $(a_i)\kin\coo$. We have
  \begin{align*}
    \Bignorm{\sum_i a_i z^{(1)}_i}_Z &\geq \frac1C \Bignorm{\sum_i
    a_i\kcdot \norm{z^{(1)}_i}_Z\kcdot v_{\min\supp_F(z^{(1)}_i)}}_V\\
    &\geq \frac1{BC} \Bignorm{\sum_i a_i\kcdot
    \norm{z^{(1)}_i}_Z\kcdot v_{k_i}}_V\\
    \intertext{and}
    \Bignorm{\sum_i a_i z^{(2)}_i}_V &\geq \frac1B \Bignorm{\sum_i a_i\kcdot
    \norm{z^{(2)}_i}_Z\kcdot v_{k_i}}_V\ .
  \end{align*}
  It follows that
  \begin{align*}
    \Bignorm{\sum_i a_iz_i}_W &\geq \frac1{BC} \max\bigg\{
    \Bignorm{\sum_i a_i\kcdot \norm{z^{(1)}_i}_Z\kcdot v_{k_i}}_V\ ,
    \Bignorm{\sum_i a_i\kcdot \norm{z^{(2)}_i}_Z\kcdot v_{k_i}}_V
    \bigg\}\\
    &\geq \frac1{2BC} \Bignorm{\sum_i a_i v_{k_i}}_V\ .
  \end{align*}
  Similarly, we have
  \begin{align*}
    \Bignorm{\sum_i a_i z^{(1)}_i}_Z &\leq C \Bignorm{\sum_i a_i\kcdot
      \norm{z^{(1)}_i}_Z\kcdot u_{\min\supp_F(z^{(1)}_i)}}_U\\
    &\leq BC \Bignorm{\sum_i a_i\kcdot \norm{z^{(1)}_i}_Z\kcdot
      u_{k_i}}_U\\
    \intertext{and}
    \Bignorm{\sum_i a_i z^{(2)}_i}_V &\leq B \Bignorm{\sum_i a_i\kcdot
    \norm{z^{(2)}_i}_Z\kcdot v_{k_i}}_V\\
    &\leq BD \Bignorm{\sum_i a_i\kcdot \norm{z^{(2)}_i}_Z\kcdot
      u_{k_i}}_U\ .
  \end{align*}
  It follows that
  \begin{align*}
    \Bignorm{\sum_i a_iz_i}_W &\leq \max\bigg\{
    BC \Bignorm{\sum_i a_i\kcdot \norm{z^{(1)}_i}_Z\kcdot u_{k_i}}_U\
    , BD \Bignorm{\sum_i a_i\kcdot \norm{z^{(2)}_i}_Z\kcdot u_{k_i}}_U
    \bigg\}\\
    &\leq B\kcdot\max\{C,D\}\kcdot \Bignorm{\sum_i a_i u_{k_i}}_U\ .
  \end{align*}
\end{proof}

The next two results show how norm estimates in a space and in its
dual are related.

\begin{prop}
  \label{prop:fdd-est-duality}
  Assume that $Z$ has an FDD $(E_i)$, and let $V$ be a space with a
  normalized and $1$-unconditional basis $(v_i)$. The following
  statements are equivalent:
  \begin{mylist}{(a)}
  \item[(a)]
    $(E_i)$ satisfies subsequential $V$-lower estimates in $Z$.
  \item[(b)]
    $(E^*_i)$ satisfies subsequential $\Vs$-upper estimates in
    $\Zs$.
  \end{mylist}
  (Here subsequential $\Vs$-upper estimates are with respect to
  $(v_i^*)$, the sequence of biorthogonal functionals to $(v_i)$).

  Moreover, if $(E_i)$ is bimonotone in $Z$, then the equivalence
  holds true if one replaces, for some $C\kgeq 1$, $V$-lower estimates
  by $C$-$V$-lower estimates in~(a) and $\Vs$-upper estimates by
  $C$-$\Vs$-upper estimates in~(b).
\end{prop}

\begin{rem}
  By duality, Proposition~\ref{prop:fdd-est-duality} holds if we
  interchange the words \emph{lower} and \emph{upper} in~(a) and~(b).
\end{rem}

\begin{proof}
  Without loss of generality we may assume that $(E_i)$ is bimonotone
  in $Z$.

  \noindent
  ``(a)$\Rightarrow$(b)'' Let $(z^*_i)$ be a normalized block sequence
  of $E^*\keq(E^*_n)$ in $\Zs$, and for each $i\kin\bn$ let
  $m_i\keq\min\supp_{E^*}(z^*_i)$. Given $(a_i)\kin\coo$, choose
  $z\kin S_Z$ with finite support with respect to $(E_n)$ such that
  $\bignorm{\sum a_iz^*_i}\keq \sum a_iz^*_i(z)$. For each $i\kin\bn$
  write
  \[
  P^E_{[\min\supp_{E^*}(z^*_i),\min\supp_{E^*}(z^*_{i+1}))} (z) =
  b_iz_i\ ,
  \]
  where $z_i\kin S_Z$ and $\abs{b_i}\kleq 1$. Since $(E_n)$ satisfies
  subsequential $C$-$V$-lower estimates in $Z$, we have $\bignorm{\sum
  b_iv_{m_i}}\kleq C\bignorm{\sum b_iz_i}\kleq C$. Hence
  \begin{align*}
    \Bignorm{\sum a_iz^*_i} &=\sum a_ib_i z^*_i(z_i)\leq \sum
    \abs{a_i}\abs{b_i}\\
    &\leq \Bignorm{\sum a_iv^*_{m_i}}\cdot \Bignorm{\sum b_iv_{m_i}}
    \leq C\Bignorm{\sum a_iv^*_{m_i}}\ ,
  \end{align*}
  as required.

  \noindent
  ``(b)$\Rightarrow$(a)'' Let $(z_i)$ be a normalized block sequence
  of $(E_n)$ in $Z$, and for each $i\kin\bn$ let
  $m_i\keq\min\supp_E(z_i)$. Given $(a_i)\kin\coo$, choose
  $(b_i)\kin\coo$ such that $\bignorm{\sum b_iv^*_{m_i}}\keq 1$ and
  $\bignorm{\sum a_iv_{m_i}}\keq \sum a_ib_i$. For each $i\kin\bn$
  there exists $z^*_i\kin S_{\Zs}$ such that $z^*_i(z_i)\keq 1$ and
  $\ran_{E^*}(z^*_i)\ksubset\ran_E(z_i)$. Since $(E^*_n)$ satisfies
  subsequential $C$-$\Vs$-upper estimates in $\Zs$, we have
  $\bignorm{\sum b_iz^*_i}\kleq C$, and hence
  \[
  \Bignorm{\sum a_iz_i}\geq \frac1C \sum a_ib_i = \frac1C
  \Bignorm{\sum a_iv_{m_i}}\ .
  \]
  This completes the proof.
\end{proof}

\begin{prop}
  \label{prop:tree-est-duality}
  Assume that $U$ is a space with a normalized, $1$-unconditional
  basis $(u_i)$ which is $D$-right-dominant for some $D\kgeq 1$, and
  that $X$ is a reflexive space which satisfies subsequential
  $C$-$U$-upper tree estimates for some $C\kgeq 1$.

  Then, for any $\ve\kge 0$, $X^*$ satisfies subsequential
  $(2CD\kplus\ve)$-$\Us$-lower tree estimates.
\end{prop}

\begin{rem}
  One might ask whether or not the converse of
  Proposition~\ref{prop:tree-est-duality} is true, \ie similar to the
  FDD case, whether $X$ satisfies subsequential $U$-upper tree
  estimates if $X^*$ satisfies subsequential $\Us$-lower tree
  estimates.

  The answer is affirmative under certain conditions on $U$, but we do
  not give a direct proof for that fact. Instead, we shall deduce it
  from one of our main embedding theorems (see
  Corollary~\ref{prop:full-tree-est-duality} in
  Section~\ref{section:embedding}).
\end{rem}

\begin{proof}
  We start with a simple observation. Let $(x^*_i)$ be a normalized,
  weakly null sequence in $X^*$. For each $n\kin\bn$ pick $x_n\kin
  S_X$ with $x^*_n(x_n)\keq 1$. There exist $y\kin X$ and $k_1\kle
  k_2\kle\dots$ in $\bn$ such that $x_{k_n}\stackrel{w}{\to} y$. Given
  $\eta\kin (0,1)$, there exists $n_0\kin\bn$ such that
  $\abs{x^*_{k_n}(y)}\kle \eta$ for all $n\kgeq n_0$. Set
  \[
  y^*_n=x^*_{k_{n_0+n}}\quad\text{and}\quad y_n =
  \frac{x_{k_{n_0+n}}-y}{\norm{x_{k_{n_0+n}}-y}}\ ,\qquad n\kin\bn\ .
  \]
  We have found, for given $\eta\kin(0,1)$, a subsequence $(y^*_i)$ of
  $(x^*_i)$ and a normalized, weakly null sequence $(y_i)$ in $X$
  satisfying $y^*_n(y_n)\kge (1\kminus\eta)/2$ for all $n\kin\bn$.

  Now let $(x^*_\alpha)_{\alpha\in\odd}$ be a normalized, weakly null
  even tree in $X^*$. By the above observation we can find a
  normalized, weakly null even tree $(y_\alpha)_{\alpha\in\odd}$ in $X$
  and a full subtree $(y^*_\alpha)_{\alpha\in\odd}$ of
  $(x^*_\alpha)_{\alpha\in\odd}$ such that $y^*_\alpha(y_\alpha)\kge
  (1\kminus\eta)/2$ for all $\alpha\kin\odd$. By a further pruning of
  these trees, we can also assume that $\abs{y^*_\alpha(y_\beta)}\kle
  \eta/2^{\max\{\abs{\alpha},\abs{\beta}\}}$ whenever
  $\alpha\kle\beta$ or $\beta\kle \alpha$.

  By assumption, there exist $m_1\kle m_2\kle\dots$ in $\bn$ such that
  $\big(y_{(m_1,m_2,\dots,m_{2i})}\big)$ is $C$-dominated
  by $(u_{m_{2i-1}})$. Given $(a_i)\kin\coo$, there exists
  $(b_i)\kin\coo$ such that $\bignorm{\sum b_iu_{m_{2i-1}}}\keq 1$ and
  $\sum a_ib_i\keq\bignorm{\sum a_iu^*_{m_{2i-1}}}$. So
  $\bignorm{\sum b_i y_{(m_1,m_2,\dots,m_{2i})}}\kleq C$, and hence
  \begin{align*}
    \Bignorm{\sum a_i y^*_{(m_1,m_2,\dots,m_{2i})}} &\geq \frac1C
    \bigg( \sum a_ib_i\frac{1\kminus\eta}2 - \sum _{i\neq j}
    \abs{a_i}\abs{b_j}\eta/ 2^{\max\{ i,j\}} \bigg)\\
    &> \frac1{2C+\ve/D} \Bignorm{\sum a_iu^*_{m_{2i-1}}}
  \end{align*}
  provided $\eta$ is sufficiently small. Now the branch
  $\big(y^*_{(m_1,m_2,\dots,m_{2i})}\big)$ of
  $\big(y^*_\alpha\big)_{\alpha\in\odd}$ corresponds to a branch
  $\big(x^*_{(n_1,n_2,\dots,n_{2i})}\big)$ of
  $\big(x^*_\alpha\big)_{\alpha\in\odd}$, where $n_1\kle n_2\kle\dots$
  and $m_i\kleq n_i$ for all $i\kin\bn$. Since $(u_i)$ is
  $D$-right-dominant, it follows that $(u^*_i)$ is $D$-left-dominant,
  and hence the above inequality shows that
  $\big(x^*_{(n_1,n_2,\dots,n_{2i})}\big)$ $(2CD\kplus\ve)$-dominates
  $\big(u^*_{n_{2i-1}}\big)$.
\end{proof}

We conclude this section with a key combinatorial result. We need to
fix some terminology first.

Given a Banach space $X$, we let $(\bn\ktimes S_X)^\omega$ denote the
set of all sequences $(k_i,x_i)$, where $k_1\kle k_2\kle\dots$ are
positive integers, and $(x_i)$ is a sequence in $S_X$. We equip the
set $(\bn\ktimes S_X)^\omega$ with the product topology of the
discrete topologies of $\bn$ and $S_X$. Given $\cA\ksubset (\bn\ktimes
S_X)^\omega$ and $\ve\kge 0$, we let
\[
\cA_\ve = \Big\{ (\ell_i,y_i)\kin (\bn\ktimes S_X)^\omega :\,\E
(k_i,x_i)\kin\cA\quad k_i\kleq\ell_i\,,\ \norm{x_i\kminus
  y_i}\kle\ve\kcdot 2^{-i}\ \V i\kin\bn\Big\}\ ,
\]
and we let $\overline{\cA}$ be the closure of $\cA$ in
$(\bn\ktimes S_X)^\omega$.

Given $\cA\ksubset (\bn\ktimes S_X)^\omega$, we say that an even tree
$(x_\alpha)_{\alpha\in \odd}$ in $X$ \emph{has a branch in $\cA$ }if
there exist $n_1\kle n_2\kle\dots$ in $\bn$ such that
$\big(n_{2i-1},x_{(n_1,n_2,\dots,n_{2i})}\big)\kin\cA$.

\begin{prop}
  \label{prop:tree-est-sb-est}
  Let $X$ be an infinite-dimensional (closed) subspace of a reflexive
  space $Z$ with an FDD $(E_i)$. Let $\cA\ksubset (\bn\ktimes
  S_X)^\omega$. Then the following are equivalent.
  \begin{mylist}{(a)}
  \item[(a)]
    For all $\ve\kge 0$ every normalized, weakly null even tree in $X$
    has a branch in $\overline{\cA_\ve}$.
  \item[(b)] 
    For all $\ve\kge 0$ there exist $(K_i)\ksubset\bn$ with $K_1\kle
    K_2\kle\dots$, $\deltab\keq (\delta_i)\ksubset (0,1)$ with
    $\delta_i\downarrow 0$, and a blocking $F\keq (F_i)$ of $(E_i)$
    such that if $(x_i)\ksubset S_X$ is a $\deltab$-skipped block
    sequence of $(F_n)$ in $Z$ with
    $\norm{x_i-P^F_{(r_{i-1},r_i)}x_i}\kle\delta_i$ for all
    $i\kin\bn$, where $1\kleq r_0\kle r_1\kle r_2\kle\dots$, then
    $(K_{r_{i-1}},x_i)\kin\overline{\cA_\ve}$.
  \end{mylist}
\end{prop}

\begin{proof}
  For each $m\kin\bn$ we set $Z_m\keq\overline{\bigoplus _{i>m}
    E_i}$. Given $\ve\kge 0$, we consider the following game between
  players S (subspace chooser) and P (point chooser). The game has an
  infinite sequence of moves; on the $n^{\text{th}}$ move ($n\kin\bn$)
  S picks $k_n,m_n\kin\bn$ and P responds by picking $x_n\kin S_X$
    with $d(x_n,Z_{m_n})\kle\ve'\kcdot 2^{-n}$, where
    $\ve'\keq\min\{\ve,1\}$. S wins the game if the sequence
    $(k_i,x_i)$ the players generate ends up in
    $\overline{\cA_{5\ve}}$, otherwise P is declared the winner. We
    will refer to this as the $(\cA,\ve)$-game and show that
    statements~(a) and~(b) above are equivalent to
  \begin{mylist}{(a)} \itshape
  \item[(c)]
    For all $\ve\kge 0$ S has a winning strategy for the
    $(\cA,\ve)$-game.
  \end{mylist}
  Note that statement~(b) yields a particular winning strategy
  for~S, so the implication (b)$\Rightarrow$(c) is clear, however this
  is included in the sequence of implications
  (a)$\Rightarrow$(c)$\Rightarrow$(b)$\Rightarrow$(a) which is what we
  are about to demonstrate.
  
  \noindent
  ``(a)$\Rightarrow$(c)'' Assume that for some $\ve\kge 0$ S does not
  have a winning strategy for the $(\cA,\ve)$-game. Then there is a
  winning strategy $\phi$ for the point chooser P. Thus $\phi$ is a
  function taking values in $S_X$ such that for all sequences $(k_i),
  (m_i)$ in $\bn$ if $x_n\keq\phi(k_1,m_1,k_2,m_2,\dots,k_n,m_n)$ for
  all $n\kin\bn$, then $d(x_i,Z_{m_i})\kle\ve'\kcdot 2^{-i}$ for all
  $i\kin\bn$ and $(k_i,x_i)\knotin\overline{\cA_{5\ve}}$. We will now
  construct a normalized, weakly null even tree
  $(x_\alpha)_{\alpha\in\odd}$ in $X$ to show that~(a) fails. This
  will be a recursive construction which also builds auxiliary trees
  $(y_\alpha)_{\alpha\in\odd}$ in $X$ and $(m_\alpha)_{\alpha\in\odd}$
  in $\bn$.

  Fix positive integers $\ell,\ n_1\kle n_2\kle\dots\kle n_{2\ell-1}$.
  Let $\alpha\keq (n_1,n_2,\dots,n_{2\ell-1})\kin
  T_\infty$, and for $1\kleq j\kleq \ell$ set $k_j\keq
  n_{2j-1}$. Assume that for $1\kleq j\kle\ell$ we have already defined
  $x_j\keq x_{(n_1,n_2,\dots,n_{2j})}$, $y_j\keq
  y_{(n_1,n_2,\dots,n_{2j})}$ and $m_j\keq
  m_{(n_1,n_2,\dots,n_{2j})}$ such that $y_j\keq
  \phi(k_1,m_1,k_2,m_2,\dots,k_j,m_j)$. We will now
  construct the nodes $\big(x_{(\alpha,n)}\big)$,
  $\big(y_{(\alpha,n)}\big)$ and $\big(m_{(\alpha,n)}\big)$. Set
  \[
  z_i = \phi (k_1,m_1,\dots,k_{\ell-1},m_{\ell-1},k_\ell,i)\ ,\qquad
  i\kin\bn\ .
  \]
  Note that $z_i\kin S_X$ and $d(z_i,Z_i)\kle \ve'\kcdot 2^{-\ell}$
  for all $i\kin\bn$. We now pass to a weakly convergent subsequence:
  there exist $i_1\kle i_2\kle \dots$ in $\bn$ and $z\kin X$ such
  that $z_{i_n}\stackrel{w}{\to}z$ as $n\to\infty$. Note that
  $\norm{z}\kleq \ve'\kcdot 2^{-\ell}$. For each $n\kin\bn$ set
  \[
  w_n\keq \frac{z_{i_n}-z}{\norm{z_{i_n}-z}}\ .
  \]
  Note that $(w_n)$ is a normalized, weakly null sequence in $X$, and
  \[
  \norm{z_{i_n}-w_n}\leq \frac{2\ve'\kcdot 2^{-\ell}}{1-\ve'\kcdot
    2^{-\ell}}\leq 4\ve\kcdot 2^{-\ell}
  \]
  for all $n\kin\bn$. We now set
  $x_{(n_1,n_2,\dots,n_{2\ell-1},n)}\keq w_n$,
  $y_{(n_1,n_2,\dots,n_{2\ell-1},n)}\keq z_{i_n}$ and
  $m_{(n_1,n_2,\dots,n_{2\ell-1},n)}\keq i_n$ for all $n\kin\bn$ with
  $n\kge n_{2\ell-1}$. This completes the recursive construction.

  It follows by induction that $(x_\alpha)_{\alpha\in\odd}$ is a
  normalized, weakly null even tree in $X$ and
  $(y_\alpha)_{\alpha\in\odd}$ is a normalized even tree in $X$ such
  that for all $\alpha\kin\odd$ we have $\norm{x_\alpha\kminus
  y_\alpha}\kleq 4\ve\kcdot 2^{-\abs{\alpha}/2}$. Moreover, given a
  sequence
  $n_1\kle n_2\kle\dots$ in $\bn$, setting $k_j\keq n_{2j-1},\ m_j\keq
  m_{(n_1,n_2,\dots,n_{2j})}$ and $y_j\keq y_{(n_1,n_2,\dots,n_{2j})}$
  for all $j\kin\bn$, we have
  \[
  y_n\keq \phi(k_1,m_1,k_2,m_2,\dots,k_n,m_n)\qquad\text{for all
  }n\kin\bn\ .
  \]
  Hence no branch of $(y_\alpha)_{\alpha\in\odd}$ is in
  $\overline{\cA_{5\ve}}$, and no branch of
  $(x_\alpha)_{\alpha\in\odd}$ is in $\overline{\cA_\ve}$.

  \noindent
  ``(c)$\Rightarrow$(b)'' Let $(\phi,\psi)$ be a winning strategy for
  $S$ in the $(\cA,\ve)$-game. Thus $\phi$ and $\psi$ are functions
  taking values in $\bn$ such that for all sequences $(k_i), (m_i)$ in
  $\bn$ and $(x_i)$ in $S_X$ if $d(x_n,Z_{m_n})\kle \ve'\kcdot 2^{-n}$,
  $k_n\keq\phi(x_1,x_2,\dots,x_{n-1})$ and
  $m_n\kgeq\psi(x_1,x_2,\dots,x_{n-1})$ for all
  $n\kin\bn$, then $(k_i,x_i)\kin\overline{\cA_{5\ve}}$. For each
  interval $I\ksubset\bn$ and $\delta\kge 0$ fix a finite set
  $S_{I,\delta}\ksubset S_X$ such that for all $x\kin S_{I,\delta}$ we
  have $\norm{x\kminus P^E_Ix}\kle\delta$ and for all $y\kin S_X$ if
  $\norm{y\kminus P^E_Iy}\kle\delta$, then there exists $x\kin
  S_{I,\delta}$ such that $\norm{x\kminus y}\kle 3\delta$.

  We now construct a blocking $(F_i)$ of $(E_i)$ by recursion. Let
  $m_1\keq\psi()$ and $F_1\keq\bigoplus _{i=1}^{m_1} E_i$. Choose any
  $m_2\kge m_1$ and set $F_2\keq\bigoplus_{i=m_1+1}^{m_2} E_i$. Assume
  that for some $n\kin\bn,\ n\kgeq 3$ , we have already chosen
  $m_1\kle\dots\kle m_{n-1}$ and we have set
  $F_j\keq\bigoplus_{i=m_{j-1}+1}^{m_j} E_i$ for $1\kleq j\kle n$
  ($m_0\keq 0$). We now choose $m_n\kge m_{n-1}$ such that if
  $\ell\kin\bn,\ 1\kleq r_0\kle r_1\kle\dots\kle r_\ell\kleq n$ and
  \[
  x_j\kin S_{[m_{r_{j-1}}+1,m_{r_j-1}],\ve'\cdot
  2^{-j}}\qquad\text{for }1\kleq j\kleq\ell\ ,
  \]
  then $m_n\kgeq\psi(x_1,x_2,\dots,x_\ell)$. Finally, we set $F_n\keq
  \bigoplus_{i=m_{n-1}+1}^{m_n} E_i$. This completes the recursive
  construction.

  For each $n\kin\bn$ let $\delta_n\keq\ve'\kcdot 2^{-n}$, and let
  $K_n$ be chosen so that $K_n\kgeq\phi()$, and if $\ell\kin\bn,\
  1\kleq r_0\kle r_1\kle\dots\kle r_\ell\kleq n$ and $x_j\kin
  S_{[m_{r_{j-1}}+1,m_{r_j-1}],\ve'\cdot 2^{-j}}$ for $1\kleq
  j\kleq\ell$,
  then $K_n\kgeq\phi(x_1,\dots,x_\ell) $. We can of course also ensure
  that the sequence $(K_i)$ is strictly increasing. Let $\deltab\keq
  (\delta_i)$. We will now verify that~(b) holds. Let $(y_i)$ be a
  $\deltab$-skipped block sequence of $(F_n)$: there exist $1\kleq
  r_0\kle r_1\kle r_2\kle\dots$ such that
  \[
  \norm{y_i-P^F_{[r_{i-1}+1,r_i-1]}y_i}<\ve'\kcdot
  2^{-i}\qquad\text{for all }i\kin\bn\ ,
  \]
  that is to say,
  \[
  \norm{y_i-P^E_{[m_{r_{i-1}}+1,m_{r_i-1}]}y_i}<\ve'\kcdot 2^{-i}
  \qquad\text{for all }i\kin\bn\ .
  \]
  For each $i\kin\bn$ there exists $x_i\kin
  S_{[m_{r_{i-1}}+1,m_{r_i-1}],\ve'\cdot 2^{-i}}$ such that
  $\norm{x_i-y_i}\kle 3\ve'\kcdot 2^{-i}$. Set
  \[
  k_i\keq\phi(x_1,\dots,x_{i-1})\qquad\text{for each }i\kin\bn\ .
  \]
  Consider the sequence $k_1, m_{r_0}, x_1, k_2, m_{r_1},
  x_2,\dots$. We have $x_i\kin S_X$ and
  \[
  d\big(x_i,Z_{m_{r_{i-1}}}\big)\leq
  \bignorm{x_i-P^E_{[m_{r_{i-1}}+1,m_{r_i-1}]}x_i} < \ve'\kcdot 2^{-i}
  \]
  for all $i\kin\bn$. Moreover $m_{r_0}\kgeq m_1\kgeq \psi()$,
  $K_{r_0}\kgeq K_1\kgeq\phi()\keq k_1$, and
  given $\ell\kin\bn$, setting $n\keq r_\ell$, we have $1\kleq r_0\kle
  r_1\kle\dots\kle r_\ell\kleq n$ and $x_i\kin
  S_{[m_{r_{i-1}}+1,m_{r_i-1}],\ve'\cdot 2^{-i}}$ for $1\kleq
  i\kleq\ell$. It follows that ($n\kgeq 3$ and) $m_{r_\ell}\keq
  m_n\kgeq \psi(x_1,\dots,x_\ell)$ and $k_{\ell+1}\kleq
  K_n\keq K_{r_\ell}$. So $(k_i,x_i)\kin\overline{\cA_{5\ve}}$, and
  hence $(K_{r_{i-1}},y_i)\kin\overline{\cA_{8\ve}}$.

  \noindent
  ``(b)$\Rightarrow$(a)'' Given $\ve\kge 0$, let $(K_i),\
  \deltab\keq(\delta_i)$ and $(F_i)$ be as in statement~(b). First
  note that if $(x_i)$ is a normalized, weakly null sequence in $X$,
  then
  \[
  \V \eta\kge 0\ \V p\kin\bn\quad \E n\kin\bn\ \E q\kge
  p\qquad\text{such that}\quad\norm{x_n-P^F_{(p,q)}x_n}\kle\eta\ .
  \]
  Indeed, the sequence $\big(P^F_{[1,p]}x_i\big)$ is weakly null, and
  hence norm-null, so there exists $n\kin\bn$ such that
  $\norm{P^F_{[1,p]}x_n}\kle\eta/2$. One can then choose $q\kge p$
  such that $\norm{P^F_{[q,\infty)}x_n}\kle\eta/2$. The claim now
  follows by triangle-inequality. 

  Now let $(x_\alpha)_{\alpha\in\odd}$ be a normalized, weakly null
  even tree in $X$. We choose positive integers $n_1\kle n_2\kle\dots$
  and $1\keq r_0\kle r_1\kle r_2\kle\dots$ by recursion. For
  $\ell\kin\bn$ we first choose $n_{2\ell-1}\kge K_{r_{\ell-1}}$ such
  that $n_{2\ell-1}\kge n_{2\ell-2}$ ($n_0\keq 0$), and then choose
  $n_{2\ell}\kge n_{2\ell-1}$ and $r_\ell\kge r_{\ell-1}$ such that
  \[
  \bignorm{x_{(n_1,n_2,\dots,n_{2\ell})}-
    P^F_{(r_{\ell-1},r_\ell)}x_{(n_1,n_2,\dots,n_{2\ell})}}
  \kle\delta_\ell\ .
  \]
  By assumption~(b) we have
  $\big(n_{2i-1},x_{(n_1,n_2,\dots,n_{2i})}\big)
  \kin\overline{\cA_\ve}$.
\end{proof}

\section{The space $Z^V(E)$}
\label{section:zv}

Let $Z$ be a space with an FDD $E=(E_i)$, and let $V$ be a space with
a $1$-unconditional and normalized basis $(v_i)$. The space
$Z^V\keq Z^V(E)$ is defined to be the completion of $\coo(\oplus E_i)$
with respect to the following norm $\norm{\cdot}_{Z^V}$.
\[
\norm{z}_{Z^V}=\max_{%
  \begin{subarray}{c}
    k\in\bn\\
    1\leq n_0<n_1<n_2<\dots<n_k
  \end{subarray}}
  \Bignorm{\sum_{j=1}^k\norm{P^E_{[n_{j-1},n_j)}(z)}_Z\kcdot
    v_{n_{j-1}}}_V\qquad\text{for }z\kin\coo(\oplus E_i)\ .
\]
Note that if $(v_i)$ is $C$-block-stable and $D$-right-dominant, then
the projection constant $K(E,Z^V)$ of $(E_i)$ in $Z^V$ satisfies
\[
K(E,Z^V)\leq \min\{ K(E,Z), C, D, 2\}\ .
\]
Here we allow $C\keq\infty$ or $D\keq\infty$ if $(v_i)$ is not
block-stable or not right-dominant, respectively. Note also that if
$\norm{\cdot}$ and $\norm{\cdot}'$ are equivalent norms on $Z$, then
the corresponding norms $\norm{\cdot}_{Z^V}$ and $\norm{\cdot}_{Z^V}'$
are equivalent on $\coo(\oplus E_i)$. This often allows us, when
examining the space $Z^V$, to assume that $(E_i)$ is bimonotone in
$Z$.

Our first set of results culminating in Corollary~\ref{cor:refl-ZV}
determine when the space $Z^V(E)$ is reflexive.
\begin{lem}
  \label{lem:bs-est-in-ZV}
  Every normalized block sequence $(z_i)$ of $(E_n)$ in $Z^V$
  $1$-dominates some block sequence $(b_i)$ of $(v_n)$ that satisfies
  \[
  1/2\kleq\norm{b_i}_V\kleq 1\quad\text{and}\quad
  \ran(b_i)\ksubset\ran_E(z_i)\qquad\text{for all }i\kin\bn.
  \]
  (Here the range, $\ran(x)$, of $x\keq\sum a_iv_i\kin V$ is the
  smallest interval in $\bn$ containing $\{i\kin\bn:\,a_i\kneq 0\}$.)
\end{lem}

\begin{proof}
  Let $z\kin S_{Z^V}$ have finite support with respect to
  $(E_i)$. Choose $k,\ 1\kleq n_0\kle n_1\kle\dots\kle n_k$ in $\bn$
  such that
  \[
  \norm{z}_{Z^V}=\Bignorm{\sum_{j=1}^k\norm{P^E_{[n_{j-1},n_j)}(z)}_Z\kcdot
    v_{n_{j-1}}}_V\ .
  \]
  Without loss of generality we can assume that $n_0\kleq
  \min\supp_E(z)\kleq n_1\kminus 1$ and that
  $n_{k-1}\kleq\max\supp_E(z)\keq n_k\kminus 1$. Set
  $m_0\keq\min\supp_E(z),\ m_j\keq n_j$ for $1\kleq j\kle k$, and let
  $m_k\kge\max\supp_E(z)$. By the triangle-inequality we have
  \begin{align*}
    \norm{z}_{Z^V} &\leq \bignorm{\norm{P^E_{[n_0,n_1)}(z)}_Z\kcdot
	v_{n_0}}_V + \Bignorm{\sum_{j=2}^k
	\norm{P^E_{[n_{j-1},n_j)}(z)}_Z\kcdot v_{n_{j-1}}}_V\\
    &\leq 2 \Bignorm{\sum_{j=1}^k\norm{P^E_{[m_{j-1},m_j)}(z)}_Z\kcdot
	    v_{m_{j-1}}}_V\ .
  \end{align*}
  Now let $(z_i)$ be a normalized block sequence of $(E_n)$ in
  $Z^V$. It follows from the above that there exist positive integers
  $1\kleq n_0\kle n_1\kle\dots$ and $1\keq k_1\kle k_2\kle\dots$ such
  that $n_{k_\ell-1}\keq\min\supp_E(z_\ell),\ n_{k_{\ell+1}-2}\kleq
  \max\supp_E(z_\ell)$, and
  \[
  \norm{z_\ell}_{Z^V}\leq 2 \Bignorm{\sum_{j=k_\ell}^{k_{\ell+1}-1}
    \norm{P^E_{[n_{j-1},n_j)}(z_\ell)}_Z\kcdot v_{n_{j-1}}}_V
  \qquad\text{for all }\ell\kin\bn\ .
  \]
  It follows that
  \[
  b_\ell=\sum_{j=k_\ell}^{k_{\ell+1}-1}
  \norm{P^E_{[n_{j-1},n_j)}(z_\ell)}_Z\kcdot v_{n_{j-1}}
  \]
  satisfies $1/2\kleq\norm{b_\ell}_V\kleq 1$ and
  $\ran(b_\ell)\ksubset\ran_E(z_\ell)$ for all
  $\ell\kin\bn$. Moreover, given $(a_i)\kin\coo$, setting
  $z\keq\sum{a_iz_i}$ we have
  \begin{align*}
    \norm{z}_{Z^V} &\geq \Bignorm{\sum_{j=1}^\infty
    \norm{P^E_{[n_{j-1},n_j)}(z)}_Z\kcdot v_{n_{j-1}}}_V \\
    &= \Bignorm{\sum_{\ell=1}^\infty \sum_{j=k_\ell}^{k_{\ell+1}-1} 
    \norm{P^E_{[n_{j-1},n_j)}(z)}_Z\kcdot v_{n_{j-1}}}_V \\
    &= \Bignorm{\sum_{\ell=1}^\infty a_\ell b_\ell}_V\ .
  \end{align*}
\end{proof}

\begin{cor}
  \label{cor:bdd-complete-ZV}
  Let $V$ be a Banach space with a normalized and $1$-unconditional
  basis $(v_i)$, and let $Z$ be a space with an FDD $E\keq(E_i)$.

  If the basis $(v_i)$ is boundedly complete, then $(E_i)$ is a
  boundedly complete FDD for $Z^V(E)$.
\end{cor}
\begin{proof}
  Let $(z_i)$ be a normalized block sequence of $(E_n)$ in $Z^V$. Let
  $(b_i)$ be a block sequence of $(v_n)$ given by
  Lemma~\ref{lem:bs-est-in-ZV}. Given $\ve\kge 0$, let $(a_i)$ be a
  scalar sequence with $\abs{a_i}\kge \ve$ for all $i\kin\bn$. Since
  $(v_i)$ is boundedly complete, and since $(z_i)$ dominates $(b_i)$
  it follows that
  \[
  \sup_n\Bignorm{\sum_{i=1}^n a_iz_i}_{Z^V} =\infty\ .
  \]
  Hence $(E_i)$ is a boundedly complete FDD for $Z^V(E)$.
\end{proof}

\begin{lem}
  \label{lem:shrinking-ZV}
  Let $V$ be a Banach space with a normalized and $1$-unconditional
  basis $(v_i)$, and assume that the space $Z$ has an FDD $E\keq
  (E_i)$.

  If the basis $(v_i)$ is shrinking and if $(E_i)$ is a shrinking FDD
  for $Z$ then $(E_i)$ is also a shrinking FDD for $Z^V(E)$.
\end{lem}

\begin{proof} 
  Without loss of generality we may assume that $(E_i)$ is bimonotone
  in $Z$.
  We first note that, given positive integers $1\kleq n_0\kle
  n_1\kle\dots$ and vectors $z^*_j\kin\bigoplus_{i\in [n_{j-1},n_j)}
  E^*_i$ with $\norm{z^*_j}_{Z^*}\kleq 1$ for each $j\kin\bn$, if
  $v^*\keq\sum_{i=1}^\infty a_i v^*_{n_{i-1}}$ converges in $V^*$ with
  $\norm{v^*}\kleq 1$, then the series $z^*\keq \sum_{i=1}^\infty a_i
  z^*_i$ converges in $(Z^V)^*$ and $\norm{z^*}_{(Z^V)^*}\kleq
  1$. Indeed, for $p\kleq q$ in $\bn$ there exists $z\kin S_{Z^V}$
  with $\supp_E(z)\ksubset[n_{p-1},n_q)$ such that
  \begin{align}
    \label{eq:shrinking-ZV;cauchy}
    \Bignorm{\sum_{j=p}^q a_j z^*_j}_{(Z^V)^*}&=\sum_{j=p}^q a_j
    z^*_j(z) \leq\sum_{j=p}^q \abs{a_j}\kcdot
    \norm{P^E_{[n_{j-1},n_j)}(z)}_Z\\
    &\leq \Bignorm{\sum_{j=p}^q a_j v^*_{n_{j-1}}}_{V^*} \cdot
    \Bignorm{\sum_{j=p}^q\norm{P^E_{[n_{j-1},n_j)}(z)}_Z\kcdot
      v_{n_{j-1}}}_{V}\notag \\
    &\leq \Bignorm{\sum_{j=p}^q a_j v^*_{n_{j-1}}}_{V^*} \cdot
    \norm{z}_{Z^V}\leq \Bignorm{\sum_{j=p}^q a_j v^*_{n_{j-1}}}_{V^*}\
    ,\notag
  \end{align}
  which implies the claim.
  Next define $K$ to be the union of the following two sets $K_1$ and
  $K_2$:
  \begin{multline*}
    K_1 = \bigg\{ \sum_{i=1}^\infty a_iz^*_i :\,1\kleq n_0\kle
    n_1\kle\dots,\quad z^*_j\kin\bigoplus_{i\in [n_{j-1},n_j)}
    E^*_i\text{ and } \norm{z^*_j}_{Z^*}\kleq 1\ \V j\kin\bn,\\
    \Bignorm{\sum_{i=1}^\infty a_iv^*_{n_{i-1}}}_{V^*}\kleq 1\bigg\}\
    ,
  \end{multline*}
  \begin{multline*}
    K_2=\bigg\{ \sum_{i=1}^\ell a_iz^*_i : \ell\kin\bn,\ 1\kleq
    n_0\kle n_1\kle\dots\kle n_\ell\kleq\infty,\quad
    z^*_j\kin\bigoplus_{i\in [n_{j-1},n_j)} E^*_i\text{ and}\\
    \norm{z^*_j}_{Z^*}\kleq 1\text{ for }1\kleq j\kleq\ell,\
    \Bignorm{\sum_{i=1}^\ell a_iv^*_{n_{i-1}}}_{V^*}\kleq 1\bigg\}\ .
  \end{multline*}
  An element $\sum_{i=1}^\ell a_iz^*_i$ of $K_2$ will also be written
  as an infinite sum $\sum_{i=1}^\infty a_iz^*_i$ by setting $a_i\keq
  0$ and $z^*_i\keq 0$ for all $i\kge\ell$. Clearly, $K$ is a
  $Z^V$-norming subset (isometrically) of $B_{(Z^V)^*}$. We claim that
  $K$ is $w^*$-compact. Indeed, for each $k\kin\bn$ let
  $y^*_k\keq\sum_{i=1}^\infty a^{(k)}_i z^*_{(k,i)}\kin K$, where for
  some (finite or infinite) sequence $1\kleq n^{(k)}_0\kle
  n^{(k)}_1\kle n^{(k)}_2\kle\dots$ in $\bn\cup\{\infty\}$ we have
  $z^*_{(k,j)}\kin\bigoplus_{i\in [n^{(k)}_{j-1},n^{(k)}_j)} E^*_i$
  and $\norm{z^*_{(k,j)}}_{Z^*}\kleq 1$ for all~$j$, and
  $\Bignorm{\sum_i a^{(k)}_iv^*_{n^{(k)}_{i-1}}}_{V^*}\kleq 1$. After
  passing to a subsequence we can assume that
  \[
  a_i=\lim _{k\to\infty}a^{(k)}_i\qquad\text{exists for all
    $i\kin\bn$, and}
  \]
  there exists $\ell\kin\bn\cup\{0,\infty\}$ such that
  \[
  \begin{array}{rl@{\qquad}l}\ds
    \lim _{k\to\infty} n^{(k)}_j &=n_j &\text{exists for }0\kleq
    j\kle\ell\ ,\\[1.5ex] \ds
    \lim_{k\to\infty}n^{(k)}_\ell &=\infty &\text{if }\ell\kin\bn\cup
    \{0\}\ ,\\[1.5ex] \ds
    \lim_{k\to\infty} z^*_{(k,j)} &=z^*_j &\text{exists (in norm)
      for }0\kleq j\kle \ell\ ,\\[1.5ex] \ds
    w^*\text{-}\lim_{k\to\infty}z^*_{(k,\ell)} &=z^*_\ell
    &\text{exists if }\ell\kin\bn\cup\{0\}\text{ , and}\\[1.5ex] \ds
    w^*\text{-}\lim_{k\to\infty} \sum_{i\in\bn,\,i>\ell}
    a^{(k)}_iz^*_{(k,i)} &= 0 &\text{if }\ell\kin\bn\cup \{0\}\ .
  \end{array}
  \]
  Consider the case when $\ell\keq\infty$. We have $1\kleq n_0\kle
  n_1\kle\dots$, $z^*_j\kin\bigoplus_{[n_{j-1},n_j)} E^*_i$ and
  $\norm{z^*_j}_{Z^*}\kleq 1$ for all $j\kin\bn$. Moreover, since
  $(v^*_i)$ is a boundedly complete basis of $V^*$, the series
  $\sum_{i=1}^\infty a_iv^*_{n_{i-1}}$ converges and
  $\Bignorm{\sum_{i=1}^\infty a_iv^*_{n_{i-1}}}_{V^*}\kleq 1$. Hence
  $z^*\keq \sum_{i=1}^\infty a_iz^*_i$ belongs to $K$. Finally, given
  $z\kin S_Z$ with finite support with respect to $(E_i)$, for
  sufficiently large $r\kin\bn$ we have
  \begin{align*}
    \lim_{k\to\infty} y^*_k(z) &=\lim_{k\to\infty} \sum_{i=1}^r
    a^{(k)}_i z^*_{(k,i)}(z)\\
    &= \sum_{i=1}^r a_iz^*_i(z) = z^*(z)\ .
  \end{align*}
  It follows that $y^*_k\stackrel{w^*}{\to}z^*$ as $k\to\infty$.
  The case $\ell\kin\bn\cup\{0\}$ is similar. We have $1\kleq n_0\kle
  n_1\kle\dots\kle n_\ell\kleq\infty,\
  z^*_j\kin\bigoplus_{[n_{j-1},n_j)} E^*_i$ and
  $\norm{z^*_j}_{Z^*}\kleq 1$ for $1\kleq j\kleq\ell$, and
  $\sum_{i=1}^\ell a_i v^*_{n_{i-1}}\kin B_{V^*}$. So
  $z^*\keq\sum_{i=1}^\ell a_iz^*_i\kin K$ and
  $y^*_k\stackrel{w^*}{\to}z^*$ as $k\to\infty$. This completes the
  proof that $K$ is $w^*$-closed.

  We deduce that $Z^V$ is embedded in $C(K)$, the space of continuous
  functions on $K$. Let $(z_i)$ be a bounded block sequence of $(E_n)$
  in $Z^V$, and let $z^*\kin K$. Using the notation as in the
  definition of $K$, if $z^*\kin K_1$, then computing as
  in~\eqref{eq:shrinking-ZV;cauchy}
  \[
  z^*(z_i) = \sum _{j,\ n_j\geq\min\supp_E(z_i)} a_jz^*_j(z_i) \leq
  \norm{z_i}_{Z^V}\cdot \Bignorm{\sum _{j, n_j\geq\min\supp_E(z_i)}
  a_j v^*_{n_{j-1}}}_{V^*}\ ,
  \]
  which converges to zero as $i\to\infty$; and if $z^*\kin K_2$, then
  for all sufficiently large values of~$i$
  \[
  z^*(z_i) =\sum_{j=1}^\ell a_j z^*_j(z_i) = a^{\phtm{*}}_\ell z^*_\ell
  (z_i)\ ,
  \]
  which converges to zero as $i\to\infty$, since $(E_i)$ is assumed a
  shrinking FDD for $Z$.

  It follows that $(z_i)$ is weakly null in $C(K)$, and thus in
  $Z^V$. Since $(z_i)$ was an arbitrary bounded block sequence in
  $Z^V$, this finishes the proof that $(E_i)$ is shrinking in $Z^V$.
\end{proof}

From Lemma~\ref{lem:shrinking-ZV} and
Corollary~\ref{cor:bdd-complete-ZV} we obtain the following result.

\begin{cor}
  \label{cor:refl-ZV}
  Assume that $V$ is a reflexive Banach space with a normalized and
  $1$-unconditional basis $(v_i)$ and that $Z$ is a space with a
  shrinking FDD $E\keq (E_i)$. Then $Z^V(E)$ is reflexive.
\end{cor}

The idea of the norm $\norm{\cdot}_{Z^V}$ is, of course, to introduce
a subsequential $V$-lower-estimate. The next lemma determines when
this is the case.

\begin{lem}
  \label{lem:ZV_has-V-lower-est}
  Let $V$ be a Banach space with a normalized and $1$-unconditional
  basis $(v_i)$, and let $Z$ be a Banach space with an FDD $E\keq
  (E_i)$.

  If, for some $C\kgeq 1$, $(v_i)$ is $C$-block stable, then $(E_i)$
  satisfies subsequential $2C$-$V$-lower estimates in $Z^V(E)$.
\end{lem}

\begin{proof}
  Let $(z_i)$ be a normalized block sequence in $Z^V(E)$, and for each
  $i\kin\bn$ let $m_i\keq\min\supp_E(z_i)$. By
  Lemma~\ref{lem:bs-est-in-ZV}, there exists a block sequence $(b_i)$
  of $(v_n)$ with $1/2\kleq\norm{b_i}_V\kleq 1$ and
  $\ran(b_i)\ksubset\ran_E(z_i)$ for all $i\kin\bn$, which is
  $1$-dominated by $(z_i)$. Since $(v_i)$ is $1$-unconditional and
  $C$-block-stable, it follows that $(b_i)$ $2C$-dominates
  $(v_{m_i})$, which proves the lemma.
\end{proof}

The final result in this section shows when subsequential $U$-upper
estimates are preserved under $Z\mapsto Z^V$.

\begin{lem}
  \label{lem:ZV-keeps-U-upper-est}
  Let $V$ and $U$ be  Banach spaces with normalized, $1$-unconditional
  and block-stable bases $(v_i)$ and $(u_i)$, respectively, and assume
  that $(v_i)$ is dominated by $(u_i)$. Let $Z$ be a Banach space with
  an FDD $(E_i)$.

  If $(E_i)$ satisfies subsequential $U$-upper estimates in $Z$, then
  $(E_i)$ also satisfies subsequential $U$-upper estimates in $Z^V$.
\end{lem}

\begin{proof}
  Choose constants $B_V, B_U, D$ and $C$ in $[1,\infty)$ such that
  $(v_i)$ is $B_V$-block-stable, $(u_i)$ is $B_U$-block-stable,
  $(v_i)$ is $D$-dominated by $(u_i)$, and $(E_i)$  satisfies
  subsequential $C$-$U$-upper estimates in $Z$. Let $K$ be the
  projection constant of $(E_i)$ in $Z$, and set $\Cb\keq B_VD\kplus
  B_UCD\kplus 2B_VDK$. We show that for any finite block sequence
  $(z_i)_{i=1}^\ell$ of $(E_n)$, and for any $k$ and $n_0\kle
  n_1\kle\dots\kle n_k$ in $\bn$ we have (putting
  $z\keq\sum_{i=1}^\ell z_i$ and $m_j\keq\min\supp_E(z_j)$ for $1\kleq
  j\kleq \ell$)
  \begin{equation}
    \label{eq:ZV-keeps-U-upper-est;aim}
    \Bignorm{\sum_{j=1}^k \norm{P^E_{[n_{j-1},n_j)}(z)}_Z\kcdot
      v_{n_{j-1}}}_V\leq \Cb\cdot \Bignorm{\sum_{i=1}^\ell
	\norm{z_i}_{Z^V}\kcdot u_{m_i}}_U\ .
  \end{equation}
  Taking then the supremum of the left side
  of~\eqref{eq:ZV-keeps-U-upper-est;aim} over all choices of $k$ and $n_0\kle
  n_1\kle\dots\kle n_k$ in $\bn$, we obtain
  \[
  \Bignorm{\sum_{i=1}^\ell z_i}_{Z^V}\leq \Cb\cdot
  \Bignorm{\sum_{i=1}^\ell \norm{z_i}_{Z^V}\kcdot u_{m_i}}_U\ ,
  \]
  and thus that $(E_i)$ satisfies subsequential $\Cb$-$U$-upper
  estimates in $Z^V$. Note that in
  proving~\eqref{eq:ZV-keeps-U-upper-est;aim} we can of course assume
  that $n_k\kleq\max\supp_E(z_\ell)\kplus 1$.

  For $i\keq 1,2,\dots,\ell$ put
  \[
  J_i=\big\{j\kin\{1,2,\dots,k\}:\,\min\supp_E(z_i)\kleq
  n_{j-1}<n_j\kleq \min\supp_E(z_{i+1})\big\}
  \]
  (with $\min\supp_E(z_{\ell+1})\keq\max\supp_E(z_\ell)\kplus 1$) and
  $J_0\keq \{1,2,\dots,k\}\ksetminus\bigcup_{i=1}^\ell J_i$.

  For  $j\keq 1,2,\dots, k $ put
  \[
  I_j =\big\{i\kin\{1,2,\dots,\ell\}:\, n_{j-1}\kleq\min\supp_E(z_i)
  \leq\max\supp_E(z_i)\kle n_j\big\}
  \]
  and $I_0\keq \{1,2,\dots,\ell\}\ksetminus\bigcup_{j=1}^k I_j$.

  Firstly, we have
  \begin{align}
    \label{eq:ZV-keeps-U-upper-est;1}
    \Bignorm{\sum_{i=1}^\ell \sum_{j\in J_i}
      \norm{P^E_{[n_{j-1},n_j)} (z_i)}_Z\kcdot v_{n_{j-1}}}_V
      &=\Bignorm{\sum_{i=1}^\ell b_i}_V\\
      \intertext{\big(where $b_i\keq\sum_{j\in J_i}
	\norm{P^E_{[n_{j-1},n_j)} (z_i)}_Z\kcdot v_{n_{j-1}}$ for
	$1\kleq i\kleq\ell$\big)}
      &\leq B_V\Bignorm{\sum_{i=1}^\ell \norm{b_i}_V\kcdot
	v_{m_i}}_V \notag\\
      &\leq B_VD \Bignorm{\sum_{i=1}^\ell \norm{z_i}_{Z^V}\kcdot
	u_{m_i}}_U\ . \notag
  \end{align}
  Secondly,
  \begin{align}
    \label{eq:ZV-keeps-U-upper-est;2}
    \biggnorm{\sum_{j\in J_0} \Bignorm{\sum_{i\in I_j} z_i}_Z\kcdot
      v_{n_{j-1}}}_V &\leq C \Bignorm{\sum_{j\in
	J_0}\norm{b_j}_U\kcdot v_{n_{j-1}}}_V\\
    \intertext{(where $b_j\keq \sum_{i\in I_j}\norm{z_i}_Z\kcdot
      u_{m_i}$ for each $j\in J_0$)}
    \leq& CD \Bignorm{\sum_{j\in J_0}\norm{b_j}_U \kcdot
      u_{n_{j-1}}}_U
    \notag\\
    \leq& CDB_U \Bignorm{\sum_{j\in J_0} b_j}_U \leq CDB_U
    \Bignorm{\sum_{i=1}^\ell \norm{z_i}_{Z^V}\kcdot u_{m_i}}_U\
    .\notag
  \end{align}
  Thirdly, given $j\kin J_0$ and $i\kin I_0$ such that
  $P^E_{[n_{j-1},n_j)}(z_i)\kneq 0$, we have either
  \begin{equation}
    \label{eq:ZV-keeps-U-upper-est;left}
    n_{j-1}<\min\supp_E(z_i)<n_j\leq \max\supp_E(z_i)\
    ,\qquad\text{or}
  \end{equation}
  \begin{equation}
    \label{eq:ZV-keeps-U-upper-est;right}
    \min\supp_E(z_i)<n_{j-1}\leq \max\supp_E(z_i)<n_j\ .
  \end{equation}
  Let $J_{0,1}$ be the set of all $j\kin J_0$ for which there exists
  an $i\kin I_0$ such that~\eqref{eq:ZV-keeps-U-upper-est;left} holds
  and let $i^1_j$ denote the unique such $i\kin I_0$. Similarly, we
  let  $J_{0,2}$ be the set of all $j\kin J_0$ for which there exists
  an $i\kin I_0$ such that~\eqref{eq:ZV-keeps-U-upper-est;right} holds
  and we denote by $i^2_j$ the unique such $i\kin I_0$. We now obtain
  \begin{align}
    \label{eq:ZV-keeps-U-upper-est;3}
    \Bignorm{\sum_{j\in J_0} \sum_{i\in I_0} \norm{P^E_{[n_{j-1},n_j)}
	&(z_i)}_Z\kcdot v_{n_{j-1}}}_V\\
    &\leq K\Bignorm{\sum_{j\in J_{0,1}} \norm{z_{i^1_j}}_Z\kcdot
      v_{n_{j-1}}}_V + K\Bignorm{\sum_{j\in J_{0,2}}
      \norm{z_{i^2_j}}_Z\kcdot v_{n_{j-1}}}_V\notag\\
    &\leq KB_V \Bignorm{\sum_{j\in J_{0,1}} \norm{z_{i^1_j}}_Z \kcdot
      v_{m_{i^1_j}}}_V + KB_V \Bignorm{\sum_{j\in J_{0,2}}
      \norm{z_{i^2_j}}_Z\kcdot v_{m_{i^2_j}}}_V \notag\\
    &\leq 2KB_VD\Bignorm{\sum_{i=1}^\ell \norm{z_i}_{Z^V}\kcdot
      u_{m_i}}_U\ .\notag
  \end{align}
  Finally, we deduce
  from~\eqref{eq:ZV-keeps-U-upper-est;1},
  \eqref{eq:ZV-keeps-U-upper-est;2}
  and~\eqref{eq:ZV-keeps-U-upper-est;3} that
  \begin{align*}
    \Bignorm{\sum_{j=1}^k&\norm{P^E_{[n_{j-1},n_j)}(z)}_Z\kcdot
      v_{n_{j-1}}}_V\\
    &\leq \Bignorm{\sum_{i=1}^\ell\sum_{j\in J_i}
      \norm{P^E_{[n_{j-1},n_j)}(z_i)}_Z\kcdot v_{n_{j-1}}}_V\\
    &\qquad+\Bignorm{\sum_{j\in J_0}\norm{P^E_{[n_{j-1},n_j)}(z)}_Z
      \kcdot v_{n_{j-1}}}_V\\
    &\leq \Bignorm{\sum_{i=1}^\ell\sum_{j\in J_i}
      \norm{P^E_{[n_{j-1},n_j)}(z_i)}_Z\kcdot v_{n_{j-1}}}_V\\
    &\qquad+\biggnorm{\sum_{j\in J_0}\Bignorm{\sum_{i\in I_j} z_i}_Z
      \kcdot v_{n_{j-1}}}_V\\
    &\qquad+\Bignorm{\sum_{j\in J_0}\sum_{i\in I_0}
      \norm{P^E_{[n_{j-1},n_j)}(z_i)}_Z\kcdot v_{n_{j-1}}}_V\\
    &\leq(B_VD+B_UCD+2B_VDK)\Bignorm{\sum_{i=1}^m
      \norm{z_i}_{Z^V}\kcdot u_{m_i}}_U\ ,
  \end{align*}
  which finishes the proof of~\eqref{eq:ZV-keeps-U-upper-est;aim}.
\end{proof}

\section{Embedding theorems}
\label{section:embedding}

In this section we will prove and deduce some consequences of

\begin{thm}
  \label{thm:embedding}
  Assume that $V$ is a Banach space with a normalized,
  $1$-unconditional and left-dominant basis $(v_i)$. Let $X$ be a
  separable, infinite-dimensional, reflexive space with subsequential
  $V$-lower tree estimates.
  \begin{mylist}{(a)}
  \item[(a)]
    For every  reflexive space $Z$ with an FDD $E\keq(E_i)$ which
    contains $X$ there is a blocking $H\keq (H_i)$ of $(E_i)$, and
    there exists $N\kin\vegtelen{\bn}$ such that $X$ naturally
    isomorphically embeds into $Z^{V_N}(H)$.
  \item[(b)]
    There is a space $\Yt$ with a bimonotone, shrinking FDD $\Gt\keq
    (\Gt_i)$, and there exists $N\kin\vegtelen{\bn}$ such that $X$ is
    a quotient of $\Yt^{V_N}(\Gt)$.
  \end{mylist}
\end{thm}

Recall that $\vegtelen{\bn}$ denotes the set of all infinite subsets
of $\bn$, and if $V$ is a Banach space with a normalized,
$1$-unconditional basis $(v_i)$, and if $N\kin\vegtelen{\bn}$, then we
write $V_N$ for the closed linear span of $\{ v_i:\,i\kin N\}$. When
we talk about subsequential $V_N$-lower estimates, etc., it is with
respect to the normalized, $1$-unconditional basis $(v_i)_{i\in N}$ of
$V_N$.

\begin{rem}
Theorem~\ref{thm:embedding} has a quantitative version. Let
$C,D\kin[1,\infty)$ and assume that $(v_i)$ is $D$-left-dominant and
that $X$ satisfies subsequential $C$-$V$-lower tree estimates.

Then for all $K\kin[1,\infty)$ there is a constant $M\keq
M(C,D,K)\kin[1,\infty)$  such that in part~(a) if $K(E,Z)\kleq K$,
then in the conclusion $X$ $M$-embeds into $Z^{V_N}(H)$. Indeed, this
follows directly from the proof. What is important is that $M$ depends
only on the constants $C,D$ and $K$.

Also, there exists a constant $L\keq L(C,D)\kin [1,\infty)$ such that
in the conclusion of part~(b) we get an onto map $\Qt\colon
\Yt^{V_N}(\Gt)\to X$ with $\norm{\Qt}\keq 1$ and $\Qt(L\kcdot
B_{\Yt^{V_N}})\ksupset B_X$. This also follows directly from the
proof. However, the proof of part~(b) uses~\cite[Lemma 3.1]{OS1},
which in turn appeals to a theorem of Zippin~\cite{Z}. The theorem of
Zippin we need here states that every separable, reflexive space
embeds isometrically into a reflexive space with an FDD. A
quantitative version of this result claims the existence of a
universal constant $K$ such that every separable, reflexive space
embeds isometrically into a reflexive space with an FDD whose
projection constant is at most~$K$. Indeed, if this wasn't true, then
for all $n\kin\bn$ there would be a ``bad'' space $X_n$, and then the
$\ell_2$-sum of the sequence $(X_n)$ would contradict Zippin's
theorem. The existence of this universal constant $K$ gives a
quantitative version of (a special case of)~\cite[Lemma 3.1]{OS1}:
there is a universal constant $\Kb$ such that every separable,
reflexive space $X$ embeds isometrically into a reflexive space $Z$
with an FDD $E\keq (E_i)$ with $K(E,Z)\kleq \Kb$ such that
$\coo(\oplus_{i=1}^\infty E_i) \cap X$ is dense in $X$. The proof of
part~(b) now really does give the quantitative version of~(b) stated
above.

The consequences of all this are quantitative analogues of
Corollaries~\ref{cor:subspace-quotient}
and~\ref{prop:full-tree-est-duality}, and of
Theorem~\ref{thm:lower-and-upper-embedding}. We shall state (without
proof) the quantitative analogue of
Theorem~\ref{thm:lower-and-upper-embedding}, and leave the reader to
formulate the analogues of Corollaries~\ref{cor:subspace-quotient}
and~\ref{prop:full-tree-est-duality}. The proofs are straightforward:
one simply needs to keep track of the various constants in the proofs
of the qualitative statements.
\end{rem}

\begin{cor}
  \label{cor:subspace-quotient}
  Assume that $V$ is a reflexive Banach space with a normalized and
  $1$-unconditional basis $(v_i)$, and that $(v_i)$ is left-dominant
  and block-stable. Let $X$ be a separable, infinite-dimensional,
  reflexive space with subsequential $V$-lower tree estimates.

  Then $X$ is a subspace of a reflexive space $Z$ with an FDD
  satisfying subsequential $V$-lower estimates and it is a quotient of
  a reflexive space $Y$ with an FDD satisfying subsequential $V$-lower
  estimates.
\end{cor}

\begin{proof}
  By a theorem of Zippin~\cite{Z} we can embed $X$ into a reflexive
  space $W$ with an FDD $E\keq(E_i)$. Using
  Theorem~\ref{thm:embedding}~(a) we find a blocking $F\keq(F_i)$ of
  $(E_i)$ and $L\kin\vegtelen{\bn}$ such that $X$ embeds into $Z\keq
  W^{V_L}(F)$.

  Theorem~\ref{thm:embedding}~(b) provides a space $\Yt$ with a
  shrinking FDD $\Gt\keq(\Gt_i)$ and $M\kin\vegtelen{\bn}$ such that
  $X$ is a quotient of $Y\keq\Yt^{V_M}(\Gt)$.

  By Corollary~\ref{cor:refl-ZV} the spaces $Z$ and $Y$ are
  reflexive. It follows from Lemma~\ref{lem:ZV_has-V-lower-est} that
  $(F_i)$ satisfies subsequential $V_L$-lower estimates in
  $Z$, and that $(\Gt_i)$ satisfies subsequential $V_M$-lower
  estimates in $Y$. The result now follows from
  Lemma~\ref{lem:norm-est-for-subseq} (with $(u_i)$ the unit vector
  basis of $U\keq\ell_1$).
\end{proof}

From Corollary~\ref{cor:subspace-quotient} and
Proposition~\ref{prop:fdd-est-duality} we deduce in certain instances
the inverse implication of Proposition~\ref{prop:tree-est-duality}.

\begin{cor}
  \label{prop:full-tree-est-duality}
  Assume that $V$ is a reflexive Banach space with a normalized,
  $1$-unconditional basis $(v_i)$, and that $(v_i)$ is left-dominant
  and block-stable.
  
  If $X$ is a separable, infinite-dimensional, reflexive space which
  satisfies subsequential $V$-lower tree estimates, then $X^*$
  satisfies subsequential $V^*$-upper tree estimates.
\end{cor}

\begin{proof}
  By Corollary~\ref{cor:subspace-quotient} $X$ is a quotient of a
  reflexive space with an FDD satisfying subsequential $V$-lower
  estimates. Hence, by Proposition~\ref{prop:fdd-est-duality}, $X^*$
  is the subspace of a reflexive space $Z$ with an FDD $(E_i)$
  satisfying subsequential $V^*$-upper estimates.

  Now let $(x_\alpha)_{\alpha\in\odd}$ be a normalized, weakly null
  even tree in $X^*$. One can recursively choose $n_1\kle
  n_2\kle\dots$ in $\bn$ such that
  \[
  \norm{x_{(n_1,n_2,\dots,n_{2i})} - P^E_{[n_{2i-1},n_{2i+1})}\big(
    x_{(n_1,n_2,\dots,n_{2i})}\big)} < 2^{-i}
  \qquad\text{for all }i\kin\bn\ .
  \]
  Set
  \[
  z_i=\frac{P^E_{[n_{2i-1},n_{2i+1})}\big(
    x_{(n_1,n_2,\dots,n_{2i})}\big)}{\norm{P^E
    _{[n_{2i-1},n_{2i+1})}\big( x_{(n_1,n_2,\dots,n_{2i})}\big)}}
    \qquad\text{for all }i\kin\bn\ .
  \]
  Then $(z_i)$ is dominated by $(v^*_{n_{2i-1}})$ since $(E_i)$
  satisfies subsequential $V^*$-upper estimates. It follows that
  $\big(x_{(n_1,n_2,\dots,n_{2i})}\big)$ is also dominated by
  $\big(v^*_{n_{2i-1}}\big)$.
\end{proof}

\begin{thm}
  \label{thm:lower-and-upper-embedding}
  Let $V$ and $U$ be reflexive Banach spaces with $1$-unconditional,
  normalized and block-stable bases $(v_i)$ and $(u_i)$,
  respectively. Further assume that $(v_i)$ is left-dominant, $(u_i)$
  is right-dominant, and that $(v_i)$ is dominated by $(u_i)$.

  If $X$ is a separable, infinite-dimensional, reflexive Banach space
  which satisfies subsequential $(V,U)$-tree estimates, then $X$ can
  be embedded into a reflexive Banach space $Z$ with an FDD $(G_i)$
  which satisfies subsequential $(V,U)$-estimates in $Z$.
\end{thm}

\begin{proof} By Proposition~\ref{prop:tree-est-duality} $X^*$
  satisfies subsequential $U^*$-lower tree estimates, and we can apply
  Corollary~\ref{cor:subspace-quotient} to deduce that $X^*$ is the
  quotient of a  reflexive space $Y^*$ with an FDD  $(E_i^*)$ ($Y^*$
  being the dual of a space $Y$ with an FDD $(E_i)$) satisfying
  subsequential $U^*$-lower estimates in $Y^*$. Thus $X$ is a subspace
  of the reflexive space $Y$ having an FDD $(E_i)$ which, by
  Proposition~\ref{prop:fdd-est-duality}, satisfies subsequential
  $U$-upper estimates in $Y$.

  Theorem~\ref{thm:embedding} part~(a) yields a blocking $F\keq (F_i)$
  of $(E_i)$ and an infinite subset $M$ of $\bn$ such that $X$ embeds
  into $Z\keq Y^{V_M}(F)$. 

  By Corollary~\ref{cor:refl-ZV} the space $Z$ is reflexive, and by
  Lemma~\ref{lem:ZV_has-V-lower-est} $(F_i)$ satisfies subsequential
  $V_M$-lower estimates in $Z$. Since $(E_i)$ satisfies subsequential
  $U$-upper estimates in $Y$, there exists $N\kin\vegtelen{\bn}$ such
  that $(F_i)$ satisfies subsequential $U_N$-upper estimates in
  $Y$. Since $(u_i)$ is right-dominant, we may assume after replacing
  $N$ if necessary that $m_i\kleq n_i$ for all $i\kin\bn$, where $m_i$
  and $n_i$ are the $i^{\mathrm{th}}$ elements of $M$ and $N$,
  respectively. Now $(v_i)_{i\in M}$ is dominated by $(u_i)_{i\in N}$,
  so by Lemma~\ref{lem:ZV-keeps-U-upper-est} $(F_i)$ also satisfies
  subsequential $U_N$-upper estimates in $Z$. Finally, since $(v_i)$
  is left-dominant, $(F_i)$ satisfies subsequential
  $(V_N,U_N)$ estimates in $Z$. An application of
  Lemma~\ref{lem:norm-est-for-subseq} completes the argument.
\end{proof}

Before proceeding to the proof of Theorem~\ref{thm:embedding} we state
the quantitative version of
Theorem~\ref{thm:lower-and-upper-embedding} as promised earlier.

\begin{thm}
  \label{thm:lower-and-upper-embedding'}
  For all $B,C,D,L,R\kin[1,\infty)$ there exist constants
  $\Cb\keq\Cb(B,D,R)$ and $K\keq K(C,L,R)$ in $[1,\infty)$ such that
  the following
  holds. Let $V$ and $U$ be reflexive Banach spaces with
  $1$-unconditional, normalized and $B$-block-stable bases $(v_i)$ and
  $(u_i)$, respectively. Further assume that $(v_i)$ is
  $L$-left-dominant, $(u_i)$ is $R$-right-dominant, and that $(v_i)$
  is $D$-dominated by $(u_i)$.

  If $X$ is a separable, infinite-dimensional, reflexive Banach space
  which satisfies subsequential $C$-$(V,U)$-tree estimates, then $X$ can
  be $K$-embedded into a reflexive Banach space $Z$ which has a
  bimonotone FDD $(G_i)$ satisfying subsequential
  $\Cb$-$(V,U)$-estimates in $Z$. \qed
\end{thm}

\begin{proof}[Proof of Theorem~\ref{thm:embedding} part~(a)]
  Choose constants $C$ and $D$ in $[1,\infty)$ such that $X$
  satisfies subsequential $C$-$V$-lower tree estimates and  $(v_i)$ is
  $D$-left-dominant. Let $K$ be the projection constant of $(E_i)$ in
  $Z$. Set
  \[
  \cA = \big\{ (k_i,x_i)\kin (\bn\ktimes S_X)^\omega:\,(v_{k_i})\text{
  is $C$-dominated by }(x_i)\big\}\ ,
  \]
  and choose $\ve\kge 0$ such that
  \[
  \overline{\cA_\ve}\subset \big\{ (k_i,x_i)\kin
  (\bn\ktimes S_X)^\omega:\,(v_{k_i})\text{ is $2CD$-dominated by
  }(x_i)\big\}\ .
  \]
  By Proposition~\ref{prop:tree-est-sb-est} there exist
  $(K_i)\ksubset\bn$ with $K_1\kle K_2\kle\dots$, $\deltab\keq
  (\delta_i)\ksubset (0,1)$ with $\delta_i\downarrow 0$, and a
  blocking $F\keq (F_i)$ of $(E_i)$ such that
  if $(x_i)\ksubset S_X$ is a $(2K\deltab)$-skipped block sequence of
  $(F_n)$ in $Z$ with
  $\norm{x_i-P^F_{(r_{i-1},r_i)}x_i}\kle2K\delta_i$ for all
  $i\kin\bn$, where $1\kleq r_0\kle r_1\kle r_2\kle\dots$, then
  $(v_{K_{r_{i-1}}})$ is $2CD$-dominated by $(x_i)$.

  It is easy to see that we can block $(F_i)$ into an FDD $G\keq
  (G_i)$ such that there exists $(e_n)\ksubset S_X$ with
  \[
  \norm{e_n-P^G_n(e_n)} < \delta_n/2K\qquad\text{for all }n\kin\bn\ .
  \]
  Let $(v''_i)$ be a subsequence of $(v_i)$ such that if
  $(x_i)\ksubset S_X$ is a $\deltab$-skipped block sequence of $(G_n)$
  in $Z$ with $\norm{x_i-P^G_{(r_{i-1},r_i)}x_i}\kle\delta_i$ for all
  $i\kin\bn$, where $1\kleq r_0\kle r_1\kle r_2\kle\dots$, then
  $(v''_{r_{i-1}})$ is $2CD$-dominated by $(x_i)$. Note that if
  $G_j\keq\bigoplus_{i=m_{j-1}+1}^{m_j} F_i,\ j\kin\bn,\ 0\keq m_0\kle
  m_1\kle m_2\kle\dots$, then $(v''_i)\keq (v_{K_{m_i}})$ will do.

  In order to continue we need the following result from~\cite{OS1},
  which is due (in a different form) to W.~B.~Johnson~\cite{J}.

  \begin{prop}
    \label{prop:x=sum-of-sb}
    Let $X$ be a Banach space which is a subspace of a reflexive space
    $Z$ with an FDD $A\keq (A_i)$ having projection constant $K$. Let
    $\etab\keq(\eta_i)\ksubset (0,1)$ with $\eta_i\downarrow 0$. Then
    there exist positive integers $N_1\kle N_2\kle\dots$ such that the
    following holds. Given positive integers $1\kleq k_0\kle
    k_1\kle\dots$ and $x\in S_X$, there exist $x_i\kin X$ and $t_i\kin
    (N_{k_{i-1}-1},N_{k_{i-1}})$ ($i\kin\bn,\ N_0\keq 0$) such that

      \begin{mylist}{(a)}
      \item[{(a)}]
	$x\keq\sum_{i=1}^\infty x_i$, and for all $i\kin\bn$ we have
	(putting $t_0\keq 0$)
      \item[{(b)}]
	either $\norm{x_i}\kle\eta_i$ or $\norm{x_i\kminus
	P^A_{(t_{i-1},t_i)} x_i}\kle\eta_i \norm{x_i}$\ ,
      \item[{(c)}]
	$\norm{x_i\kminus P^A_{(t_{i-1},t_i)} x}\kle\eta_i$\ ,
      \item[{(d)}]
	$\norm{x_i}\kle K\kplus 1$\ ,
      \item[{(e)}]
	$\norm{P_{t_i}^A x}\kle\eta_i$\ .
      \end{mylist}
  \end{prop}
  This result is in fact a slight variation of (and follows easily
  from the proof of) Corollary~4.4 in~\cite{OS1}.
  
  We now apply Proposition~\ref{prop:x=sum-of-sb} with $A\keq G$ and
  $\etab\keq\deltab$ to obtain an appropriate sequence $N_1\kle
  N_2\kle\dots$ of positive integers. Set
  $H_j\keq\bigoplus_{i=N_{j-1}+1}^{N_j} G_i$ for each $j\kin\bn$ (and
  with $N_0\keq 0$), and let $(v'_i)$ be the subsequence of $(v_i)$
  defined by $v'_i\keq v''_{N_i}$ for all $i\kin\bn$. Let
  $N\kin\vegtelen{\bn}$ be chosen such that $(v_i)_{i\in N}$ is the
  subsequence $(v'_i)$ of $(v_i)$.

  Fix $x\kin S_X$ and a sequence $1\kleq n_0\kle n_1\kle\dots$ in
  $\bn$. We will show that
  \begin{equation}
    \label{eq:embedding(a);aim}
    \Bignorm{\sum _{i=1}^\infty \norm{P^H_{[n_{i-1},n_i)}(x)}_Z\kcdot
    v'_{n_{i-1}}}_V \leq 4KD^2C(K+2\Delta+2) + K(K+1) + 3K\Delta\ ,
  \end{equation}
  where $\Delta\keq\sum_{i=1}^\infty\delta_i$. Taking then the
  supremum over all choices of $(n_i)$, we obtain that the norms
  $\norm{\cdot}_Z$ and $\norm{\cdot}_{Z^{V_N}(H)}$ are equivalent when
  restricted to $X$, and hence statement~(a) follows.

  Set $M_i\keq N_{n_i-1}$ for $i\keq 0,1,2,\dots$. We thus have to
  show that
  \[
  \Bignorm{\sum _{i=1}^\infty \norm{P^G_{(M_{i-1},M_i]}(x)}_Z\kcdot
  v'_{n_{i-1}}}_V \leq 4KD^2C(K+2\Delta+2) + K(K+1) + 3K\Delta\ .
  \]
  For each $i\kin\bn$ choose $x_i\kin X$ and
  $t_i\kin(M_{i-1},N_{n_{i-1}})$ such that (a)--(e) of
  Proposition~\ref{prop:x=sum-of-sb} hold with $A\keq G$ and
  $\etab\keq\deltab$.

  For each $i\kin\bn$ let $\xb_i\keq \frac{x_{i+1}}{\norm{x_{i+1}}}$
  and $\alpha_i\keq\norm{x_{i+1}}$ if
  $\norm{x_{i+1}}\kgeq\delta_{i+1}$, and let $\xb_i\keq e_{M_i}$ and
  $\alpha_i\keq 0$ if $\norm{x_{i+1}}\kle\delta_{i+1}$. Observe that
  $\norm{\xb_i\kminus P^G_{(t_i,t_{i+1})}(\xb_i)}\kle \delta_i$ for all
  $i\kin\bn$, from which it follows that $(v''_{t_i})$ is
  $2CD$-dominated by  $(\xb_i)$. Hence

  \begin{align*}
    \Bignorm{\sum_{i=1}^{\infty}x_i}_Z\geq &\Bignorm{\sum_{i=1}^\infty
      \alpha_i\xb_i}_Z-\norm{x_1}_Z-\Delta\\
    \geq &\frac{1}{2CD}\Bignorm{\sum_{i=1}^\infty \alpha_i
      v''_{t_i}}_V-(K+1)-\Delta\notag\\
    \geq &\frac{1}{2CD}\Bignorm{\sum_{i=1}^\infty \norm{x_{i+1}}_Z\kcdot
      v''_{t_i}}_V- \frac{1}{2CD}\Delta-(K+1)-\Delta\notag\ ,
  \end{align*}
  and thus
  \begin{equation}
    \label{eq:embedding(a);dominate-vi}
    \Bignorm{\sum_{i=1}^\infty \norm{x_{i+1}}_Z\kcdot v''_{t_i}}_V\leq
    2CD(K+2\Delta+2)\ .
  \end{equation}
  
  For each $i\kin\bn$ we have (putting $t_0\keq 0$)
  \[
  \norm{P^G_{(M_{i-1},M_i]}(x)}_Z \leq
  K\norm{P^G_{(t_{i-1},t_{i+1})}(x)}_Z \leq K\big(
  \norm{x_i}_Z+\norm{x_{i+1}}_Z+3\delta_i\big)\ .
  \]
  It follows that
  \begin{align*}
    \Bignorm{\sum_{i=1}^\infty \norm{&P^G_{(M_{i-1},M_i]}(x)}_Z\kcdot
    v'_{n_{i-1}}}_V\\
    & \leq K\Bignorm{\sum_{i=1}^\infty \norm{x_i}_Z\kcdot
    v'_{n_{i-1}}}_V
    + K\Bignorm{\sum_{i=1}^\infty \norm{x_{i+1}}_Z\kcdot
      v'_{n_{i-1}}}_V + 3K\Delta\\
    &\leq K\Bignorm{\sum_{i=1}^\infty \norm{x_{i+1}}_Z\kcdot
      v'_{n_i}}_V
    + K\Bignorm{\sum_{i=1}^\infty \norm{x_{i+1}}_Z\kcdot
      v'_{n_{i-1}}}_V + K(K+1) + 3K\Delta\\
    & \leq 2KD\Bignorm{\sum_{i=1}^\infty \norm{x_{i+1}}_Z\kcdot
      v''_{t_i}}_V + K(K+1)+3K\Delta\\
    \intertext{(since $v'_{n_{i-1}}\keq v''_{N_{n_{i-1}}}$ and
      $N_{n_{i-1}}\kge t_i$ for all $i\kin\bn$)}\\
    &\leq 4KD^2C(K+2\Delta+2) + K(K+1) + 3K\Delta\ .
  \end{align*}
\end{proof}

Before we prove part~(b) of Theorem~\ref{thm:embedding} we need a
blocking result due to Johnson and Zippin.

\begin{prop}
  \label{prop:jz-blocking} \cite{JZ1}
  Let $T\colon Y\to Z$ be a bounded linear operator from a space $Y$
  with a shrinking FDD $(G_i)$ into a space $Z$ with an FDD
  $(H_i)$. Let $\ve_i\downarrow 0$. Then there exist blockings
  $E\keq(E_i)$ of $(G_i)$ and $F\keq(F_i)$ of $(H_i)$ so that for all
  $m\kle n$ and $y\kin S_{\bigoplus_{i\in (m,n)} E_i}$ we have
  $\norm{P^F_{[1,m)} Ty}\kle \ve_m$ and $\norm{P^F_{[n,\infty)}
    Ty}\kle\ve_n$.
\end{prop}

\begin{proof}[Proof of Theorem~\ref{thm:embedding} part~(b)]
  By Lemma 3.1 in~\cite{OS1} we can, after renorming $X$ if necessary,
  regard $X^*$ (isometrically) as a subspace of a reflexive space
  $Y^*$ (being the dual of a reflexive space $Y$ with bimonotone FDD
  $(E_i)$) such that $\coo(\oplus_{i=1}^\infty E^*_i)\cap X^*$ is
  dense in $X^*$. We have a natural quotient map $Q\colon Y\to X$. By
  a theorem of Zippin~\cite{Z} we may regard $X$ (isometrically) as a
  subspace of a reflexive space $Z$ with an FDD $(F'_i)$. Let $K$ be
  the projection constant of $(F'_i)$ in $Z$, and choose constants $C$
  and $D$ in $[1,\infty)$ such that $X$ satisfies subsequential
  $C$-$V$-lower tree estimates and $(v_i)$ is $D$-left-dominant.

  Using Proposition~\ref{prop:tree-est-sb-est} as in the proof of
  part~(a), we find sequences $(K_i)\ksubset\bn$ with $K_1\kle
  K_2\kle\dots$, $\deltab\keq (\delta_i)\ksubset (0,1)$ with
  $\delta_i\downarrow 0$, and a blocking $(F_i)$ of $(F'_i)$ such that
  if $(x_i)\ksubset S_X$ is a
  $2K\deltab$-skipped block sequence of $(F_n)$ in $Z$ with
  $\norm{x_i-P^F_{(r_{i-1},r_i)}x_i}_Z\kle2K\delta_i$ for all
  $i\kin\bn$, where $1\kleq r_0\kle r_1\kle r_2\kle\dots$, then
  $(v_{K_{r_{i-1}}})$ is $2CD$-dominated by $(x_i)$, and moreover,
  using standard perturbation arguments and making $\deltab$ smaller
  if necessary, we can assume that if $(z_i)\ksubset Z$ satisfies
  $\norm{x_i\kminus z_i}_Z\kle \delta_i$ for all $i\kin\bn$, then
  $(z_i)$ is a basic sequence equivalent to $(x_i)$ with projection
  constant at most $2K$. We also require that
  \begin{equation}
    \label{eq:embedding(b);Delta}
    \Delta=\sum_{i=1}^{\infty}\delta_i<\frac17\ .
  \end{equation}
  Choose a sequence $\veb\keq (\ve_i)\ksubset (0,1)$ with
  $\ve_i\downarrow 0$ and
  \begin{equation}
    \label{eq:embedding(b);epsilon}
    3K(K+1)\sum_{j=i}^{\infty}\ve_j<\delta_i^2\qquad\text{for all
    }i\kin\bn\ .
  \end{equation}
  After blocking $(F_i)$ if necessary, we can assume
  that for any subsequent blocking $D$ of $F$ there is a sequence
  $(e_i)$ in $S_X$ such that
  \begin{equation}
    \label{eq:embedding(b);gaps-in-sb}
    \norm{e_i-P^D_i(e_i)}_Z<\ve_i/2K\qquad\text{for all } i\kin\bn\ .
  \end{equation}
  By Proposition~\ref{prop:jz-blocking} we may assume, after further
  blocking our FDDs if necessary, that
  \begin{align}
    \label{eq:embedding(b);almost-diagonal}
    &\text{for all $m\kle n$ and $y\kin S_{\oplus_{i\in (m,n)} E_i}$
      we have}\\
    &\qquad \norm{P^F_{[1,m)}\circ Q(y)}\kle\ve_m\qquad\text{and}
      \qquad \norm{P^F_{[n,\infty)}\circ Q(y)}\kle\ve_n\ ,\notag
  \end{align}
  and moreover the same holds if one passes to any blocking of $(E_i)$
  and the corresponding blocking of $(F_i)$.

  For $i\kin\bn$ let $\Et_i$ be the quotient space of $E_i$ determined
  by $Q$, i.e., if $y\kin E_i$, then the norm of $\yt$, the
  equivalence class of $y$ in $E_i$, is given by
  $\tnorm{\yt}\keq\norm{Q(y)}$. Passing to a further blocking of
  $(E_i)$ (and the corresponding blocking of $(F_i)$), we may assume
  that $\Et_i\kneq\{0\}$ for all $i\kin\bn$. Given $y\keq\sum y_i \kin
  \coo(\oplus_{i=1}^\infty E_i)$, $y_i\kin E_i$ for all $i\kin\bn$, we
  set $\yt\keq\sum \yt_i\kin \coo(\oplus_{i=1}^\infty \Et_i)$ and
  \[
  \tnorm{\yt}=\max_{m<n} \Bignorm{\sum_{i=m}^n Q(y_i)}=\max_{m<n}
  \norm{Q\circ P^E_{[m,n]}(y)}\ .
  \]
  We let $\Yt$ be the completion of $\coo(\oplus_{i=1}^\infty \Et_i)$
  with respect to $\tnorm{\cdot}$. Since $(E_i)$ is a bimonotone FDD
  in $Y$, we have $\tnorm{\yt}\kleq\norm{y}$ for all
  $y\kin\coo(\oplus_{i=1}^\infty E_i)$, and hence the map
  $y\mapsto\yt$ extends to a norm one map from $Y$ to $\Yt$. By the
  definition of $\tnorm{\cdot}$ we have $\norm{Qy}\kleq\tnorm{\yt}$
  for any $y\kin\coo(\oplus_{i=1}^\infty E_i)$. It follows that
  $\yt\mapsto Q(y)$ extends to a norm one map $\Qt\colon \Yt\to X$
  with $\Qt(\yt)\keq Q(y)$ for all $y\kin Y$.

  In order to continue our proof we will need the following proposition
  from~\cite{OS2}.

  \begin{prop}
    \label{prop:quotient-norm}
    \cite[Proposition 2.6]{OS2}
    \begin{mylist}{(a)}
    \item[{(a)}]
      $(\Et_i)$ is a bimonotone, shrinking FDD for $\Yt$.
    \item[{(b)}]
      $\Qt$ is a quotient map from $\Yt$ onto $X$. More precisely if
      $x\kin X$ and $y\kin Y$ is such that $Q(y)\keq x$, $\norm{y}\keq
      \norm{x}$ and $y\keq\sum y_i$ with $y_i\kin E_i$ for all
      $i\kin\bn$, then $\yt\keq\sum \yt_i\kin\Yt$, $\tnorm{\yt}\keq
      \norm{y}$ and $\Qt(\yt)\keq x$.
    \item[{(c)}]
      Let $(\yt_i)$ be a block sequence of $(\Et_n)$ in $B_{\Yt}$,
      and assume that $(\Qt(\yt _i))$ is a basic sequence with
      projection constant $\Kb$ and that $a\keq\inf_i
      \norm{\Qt(\yt_i)}\kge 0$. Then for all  $(a_i)\kin\coo$ we have
      \begin{equation*}
	\Bignorm{\sum a_i \Qt(\yt _i)}\leq \Bigtnorm{\sum a_i\yt_i}
	\leq \frac{3\Kb}{a} \Bignorm{\sum a_i \Qt(\yt _i)}\ .
      \end{equation*}
    \end{mylist}
  \end{prop}

  To finish the proof of Theorem~\ref{thm:embedding}~(b) it suffices
  to find a constant $L\kle\infty$, a subsequence $(v'_i)$ of $(v_i)$,
  and a blocking $\Gt\keq(\Gt_i)$ of $(\Et_i)$ with the following
  property. For each $x\kin S_X$ there exists a $\yt\keq\sum \yt_i\kin
  \Yt$, $\yt_i\kin \Gt_i$ for all $i\kin\bn$, such that
  \begin{align}
    \label{eq:embedding(b);new-quotient1}
    &\norm{\Qt(\yt)-x}<1/2\ ,\\
    &\label{eq:embedding(b);new-quotient2}
    \Bignorm{\sum_{j=1}^\infty\tnorm{P^{\Gt}_{[n_{j-1},n_j)}(\yt)}
      \kcdot v'_{n_{j-1}}}_V\leq L\\
    &\ \text{for any choice of }k\ \text{and}\ 1\kleq n_0\kle n_1\kle
      n_2\kle\dots\text{ in }\bn\ .\notag
  \end{align}
  Once this is accomplished, we consider the space
  $\Yt^{V_N}\keq \Yt^{V_N}(\Gt)$, where $N\kin\vegtelen{\bn}$ is
  chosen so that $(v_i)_{i\in N}$ is the subsequence $(v'_i)$ of
  $(v_i)$. Given $x\keq x_0\kin S_X$, the property of
  $\Gt$ allows us to recursively choose $x_n\kin\frac1{2^n}B_X$ and
  $\yt_n\kin \frac{L}{2^{n-1}}B_{\Yt^{V_N}}$, $n\kin\bn$, so that
  $x_n\keq
  x_{n-1}\kminus\Qt(\yt_n)$ for all $n\kin\bn$. It follows that
  $\sum_{n=1}^\infty \yt_n $ converges in $\Yt^{V_N}$ with
  $\norm{\sum_{n=1}^\infty \yt_n}_{\Yt^{V_N}}\kleq 2L$ and
  $\Qt(\sum_{n=1}^\infty \yt_n)\keq x$. Thus $\Qt\colon\Yt^{V_N}\to X$
  remains surjective, which finishes the proof.

  In order to show the existence of a suitable blocking $\Gt$ of $\Et$
  we need the following result from~\cite{OS2}.
  
  \begin{lem}
    \label{lem:kill-the-overlap} \cite[Lemma  2.7]{OS2}
    Assume that~\eqref{eq:embedding(b);almost-diagonal} holds for our
    original map $Q\colon Y\to X$. Then there exist integers $0\keq
    N_0\kle N_1\kle\dots$ so that if for each $i\kin\bn$ we define
    \begin{equation*}
      \begin{split}
	& C_i=\bigoplus_{j=N_{i-1}+1}^{N_i} E_j\ ,\qquad
	D_i=\bigoplus_{j=N_{i-1}+1}^{N_i} F_j,\\
	& L_i = \left\{ j\kin\bn:\,N_{i-1}\kle
	j\kleq\frac{N_{i-1}+N_i}2\right\}\ ,\\
	& R_i = \left\{ j\kin\bn:\,\frac{N_{i-1} +
	  N_i}2\kle j\kleq N_i \right\}\ ,\\
	& C_{i,L} = \bigoplus_{j\in L_i} E_j\qquad\text{and}\qquad
	C_{i,R} = \bigoplus_{j\in R_i} E_j \ ,
      \end{split}
    \end{equation*}
    then the following holds. Let $x\kin S_X$, $0\kleq m\kle n$ and
    $\ve\kge 0$, and assume that $\norm{x\kminus
    P^D_{(m,n)}(x)}\kle\ve$. Then there exists $y\kin B_Y$ with $y\kin
    C_{m,R}\oplus \Big(\bigoplus_{i\in (m,n)} C_i\Big)\oplus C_{n,L}$
    (where $C_{0,R}=\{0\}$) and $\norm{Qy\kminus x}\kle K(2\ve
    +\ve_{m+1})$ (recall that $K$ is the projection constant of
    $(F'_i)$ in $Z$).
  \end{lem}

  Let $(C_i)$ and $(D_i)$ be the blockings given by
  Lemma~\ref{lem:kill-the-overlap}. Note that the sequence $(N_i)$ in
  the lemma used to define these blockings will not be needed in the
  sequel, so we can discard it. We now apply
  Proposition~\ref{prop:x=sum-of-sb} with $(A_i)\keq(D_i)$ and
  $\etab\keq\veb$ to obtain a sequence $N_1\kle N_2\kle\dots$ in $\bn$
  so that the conclusions of the proposition are satisfied. Let
  $(v''_i)$ be a subsequence of $(v_i)$ such that if $(x_i)\ksubset
  S_X$ is a $\deltab$-skipped block sequence of $(D_n)$ in $Z$ with
  $\norm{x_i-P^D_{(r_{i-1},r_i)}x_i}_Z\kle\delta_i$ for all $i\kin\bn$,
  where $1\kleq r_0\kle r_1\kle r_2\kle\dots$, then $(v''_{r_{i-1}})$
  is $2CD$-dominated by $(x_i)$, and moreover, if $(z_i)\ksubset Z$
  satisfies $\norm{x_i\kminus z_i}\kle \delta_i$ for all $i\kin\bn$,
  then $(z_i)$ is a basic sequence equivalent to $(x_i)$ with
  projection constant at most $2K$. Let $(v'_i)$ be the subsequence of
  $(v_i)$ defined by setting $v'_i\keq v''_{N_i}$ for all
  $i\kin\bn$. We now come to our final blockings: for each $i\kin\bn$
  set $G_i\keq\bigoplus_{j=N_{i-1}+1}^{N_i} C_j$ and let
  $H_i\keq\bigoplus _{j=N_{i-1}+1}^{N_i} D_j$ ($N_0=0$). Put
  $G\keq(G_i)$, let $\Gt\keq(\Gt_i)$ be the corresponding blocking of
  $(\Et_i)$, and set $H\keq(H_i)$.

  Fix a sequence $(e_i)$ in $S_X$ so
  that~\eqref{eq:embedding(b);gaps-in-sb} holds. Let $x\kin S_X$. By
  the choice of $N_1, N_2, \dots$, for each $i\kin\bn$,  there are
  $x_i\kin (K+1)B_X$ and $t_i\kin(N_{i-1},N_i)$ such that
  $x\keq\sum_{i=1}^\infty x_i$ and for all $i\kin\bn$ either
  $\norm{x_i}\kle\ve_i$ or $\norm{P^D_{(t_{i-1},t_i)} x_i\kminus
  x_i}\kle\ve_i \norm{x_i}$ ($t_0\keq 0$). For each $i\kin\bn$ let
  $\xb_i\keq\frac{x_{i+1}}{\norm{x_{i+1}}}$ and
  $\alpha_i\keq\norm{x_{i+1}}$ if $\norm{x_{i+1}}\kgeq \ve_{i+1}$, and
  let $\xb_i\keq e_{N_i}$ and $\alpha_i\keq 0$ if
  $\norm{x_{i+1}}\kle\ve_{i+1}$.

  Since
  \begin{equation}
    \label{eq:embedding(b);xb-sb}
    \norm{\xb_i-P^D_{(t_i,t_{i+1})}(\xb_i)}<\ve_{i+1}\qquad\text{for
      all}\ i\kin\bn\ ,
  \end{equation}
  there exists $(y_i)\ksubset B_Y$ with $y_i\kin C_{t_i,R}\oplus
  \Big(\bigoplus_{j\in (t_i,t_{i+1})} C_j\Big)\oplus C_{t_{i+1},L}$
  and
  \begin{equation}
    \label{eq:embedding(b);lift-xb}
    \norm{Q(y_i) -\xb_i}< 3K\ve_{i+1}\ ,\qquad i\kin\bn\ .
  \end{equation}
  Also, if $\norm{x_1}\kle\ve_1$, then set $y_0\keq 0$, and if
  $\norm{x_1}\kgeq \ve_1$, then choose $y_0\kin (K+1)B_Y$ such that
  $y_0\kin \Big(\bigoplus_{j\in (0,t_1)} C_j\Big)\oplus
  C_{t_1,L}\ksubset G_1$ and $\norm{Q(y_0)\kminus x_1}\kle 3K(K\kplus
  1)\ve_1$.

  Set $\xb\keq x_1\kplus\sum_{i=1}^{\infty} \alpha_i\xb_i$, and note
  that (this series converges and) by~\eqref{eq:embedding(b);Delta}
  and~\eqref{eq:embedding(b);epsilon}
  \begin{equation}
    \label{eq:embedding(b);xb-close-to-x}
    \norm{x-\xb}\leq\sum_{i=2}^{\infty}\ve_i<\frac14\ .
  \end{equation}
  As a $\deltab$-skipped block sequence of $(D_i)$
  (this follows from~\eqref{eq:embedding(b);xb-sb}
  and~\eqref{eq:embedding(b);epsilon}), $(\xb_i)$ is a basic sequence
  with projection constant at most $2K$ that $2CD$-dominates
  $(v''_{t_i})$. Since, by~\eqref{eq:embedding(b);lift-xb},
  $\norm{\Qt(\yt_i)\kminus\xb_i}\kle 3K\ve_{i+1}\kle\delta_i$ for all
  $i\kin\bn$, the sequence $\big(\Qt(\yt_i)\big)$ is also a basic
  sequence with projection constant at most $2K$ and is equivalent to
  $(\xb_i)$. Furthermore, we have $\inf_i\norm{\Qt(\yt_i)}\kgeq \inf_i
  \big(\norm{\xb_i}\kminus\delta_i\big)\kge 6/7$, and thus, by
  Proposition~\ref{prop:quotient-norm}~(c),
  \begin{equation}
    \label{eq:embedding(b);yt-equiv-to-xb}
    \Bignorm{\sum a_i \Qt(\yt _i)} \leq \Bigtnorm{\sum a_i\yt _i}
    \leq 7K \Bignorm{\sum a_i \Qt(\yt _i)}\qquad\text{for all}\
    (a_i)\kin\coo\ .
  \end{equation}
  Thus $(\yt_i)$ is a basic sequence equivalent to $(\xb_i)$ and, in
  particular, $\sum_{i=1}^{\infty} \alpha_i\yt_i$ converges. Putting
  $\yt\keq\yt_0\kplus\sum_{i=1}^{\infty} \alpha_i\yt_i$ we have
  \begin{align*}
    \norm{\Qt\yt-\xb} &\leq \norm{\Qt\yt_0-x_1}+ \sum_{i=1}^{\infty}
    \abs{\alpha_i}\kcdot \norm{\Qt\yt_i-\xb_i}\\
    &\leq 3K(K+1)\sum_{i=1}^{\infty}\ve_i<1/4\ ,
  \end{align*}
  and hence, by~\eqref{eq:embedding(b);xb-close-to-x},
  $\norm{\Qt\yt-x}<1/2$, so we
  have~\eqref{eq:embedding(b);new-quotient1}.

  We now fix integers $1\kleq n_0\kle n_1\kle n_2\kle\dots$. We have
  $\yt_i\kin\Gt_i\oplus\Gt_{i+1}$ for each $i\kin\bn$, and
  $\yt_0\kin\Gt_1$. It follows that
  \begin{multline}
    \label{eq:embedding(b);new-quotient3}
    \Bignorm{\sum_{s=1}^\infty\tnorm{P^{\Gt}_{[n_{s-1},n_s)}(\yt)}
      \kcdot v'_{n_{s-1}}}_V \leq \tnorm{\yt_0}+
    \biggnorm{\sum_{s=1}^\infty \alpha_{n_{s-1}-1} \cdot
      v'_{n_{s-1}}}_V\\
    + \biggnorm{\sum_{s=1}^\infty\Bigtnorm{\sum _{i=n_{s-1}}^{n_s-1}
	\alpha_i\yt_i} \kcdot v'_{n_{s-1}}}_V\ ,
  \end{multline}
  where we put $\alpha_0\keq 0$ in case $n_0\keq 1$.
  We now show how to bound each of the three terms of the right-hand
  side of the above inequality, and hence
  obtain~\eqref{eq:embedding(b);new-quotient2} with $L\keq
  126CD^2K^3$.

  We already have $\tnorm{\yt_0}\kleq K\kplus 1$. Since $(\xb_i)$
  $2CD$-dominates $(v''_{t_i})$ we get
  \begin{align*}
    2CD^2\Bignorm{\sum_i\alpha_i\xb_i}_Z &\geq D\Bignorm{\sum_i
      \alpha_i \kcdot v''_{t_i}}_V\\
    &\geq \Bignorm{\sum_i\alpha_i \kcdot v'_{i+1}}_V\\
    \intertext{(since $v'_{i+1}\keq v''_{N_{i+1}}$ and $N_{i+1}\kge
      t_i$ for all $i\kin\bn$)}
    &\geq \Bignorm{\sum_{s=1}^\infty \alpha_{n_{s-1}-1} \cdot
      v'_{n_{s-1}}}_V\ .
  \end{align*}
  Moreover, it follows from~\eqref{eq:embedding(b);xb-close-to-x} that
  \begin{equation}
    \label{eq:embedding(b);sum-alpha-xb}
    \Bignorm{\sum_i\alpha_i\xb_i}_Z=\norm{\xb-x_1}_Z\leq K+3\ .
  \end{equation}
  This yields the bound of $2CD^2(K\kplus 3)$ for the second term
  of~\eqref{eq:embedding(b);new-quotient3}.

  For each $s\kin\bn$ let $\wt_s\keq\sum_{i=n_{s-1}}^{n_s-1}
  \alpha_i\yt_i$ and $b_s\keq\sum_{i=n_{s-1}}^{n_s-1}
  \alpha_i\xb_i$. Note that by~\eqref{eq:embedding(b);xb-sb}
  and~\eqref{eq:embedding(b);epsilon}
  \begin{align}
    \label{eq:embedding(b);bs-sb}
    \bignorm{b_s-P^D_{(t_{n_{s-1}},t_{n_s})}(b_s)} &\leq
    \sum_{i=n_{s-1}}^{n_s-1} \abs{\alpha_i}\kcdot 2K\kcdot
    \norm{\xb_i-P^D_{(t_i,t_{i+1})} \xb_i}\\
    &<2K(K+1)\sum_{i=n_{s-1}}^{n_s-1} \ve_{i+1}<\delta_s^2 &
    \text{for all }s\kin\bn\ .\notag
  \end{align}
  For each $s\kin\bn$ set $\bbar_s\keq\frac{b_s}{\norm{b_s}}$ and
  $\beta_s\keq\norm{b_s}$ if $\norm{b_s}\kgeq\delta_s$, and set
  $\bbar_s\keq\xb_{n_{s-1}}$ and $\beta_s\keq 0$ if
  $\norm{b_s}\kle\delta_s$. It follows
  from~\eqref{eq:embedding(b);bs-sb} and~\eqref{eq:embedding(b);xb-sb}
  that $(\bbar_s)\ksubset S_X$ is a $\deltab$-skipped block sequence
  of $(D_i)$ in $Z$ with $\norm{\bbar_s\kminus
  P^D_{(t_{n_{s-1}},t_{n_s})}(\bbar_s)}\kle\delta_s$ for all
  $s\kin\bn$, and hence it is a basic sequence that $2CD$-dominates
  $(v''_{t_{n_{s-1}}})$.

  From~\eqref{eq:embedding(b);lift-xb}
  and~\eqref{eq:embedding(b);yt-equiv-to-xb} we have
  \begin{align}
    \label{eq:embedding(b);wt-lifts-b}
    \norm{\Qt(\wt_s)-b_s} \leq& \sum_{i=n_{s-1}}^{n_s-1}
    \abs{\alpha_i}\kcdot \norm{\Qt(\yt_i)-\xb_i}\\
    <&3K(K+1)\sum_{i=n_{s-1}}^{n_s-1}\ve_{i+1}
    <\delta_s\notag\\
    \intertext{and}
    \label{eq:embedding(b);wt-equiv-to-b}
    \tnorm{\wt_s}\leq& 7K\norm{\Qt(\wt_s)} &\text{for all }s\kin\bn\
    .
  \end{align}

  We now obtain the following sequence of inequalities.
  \begin{align*}
    \Bignorm{\sum _{s=1}^\infty \tnorm{\wt_s}\kcdot v'_{n_{s-1}}}_V
    &\leq 7K \Bignorm{\sum _{s=1}^\infty \norm{\Qt(\wt_s)}\kcdot
      v'_{n_{s-1}}}_V
    &\text{(from~\eqref{eq:embedding(b);wt-equiv-to-b})}\\
    & \leq 7K \Bignorm{\sum _{s=1}^\infty \norm{b_s}\kcdot
      v'_{n_{s-1}}}_V+7K\Delta
    &\text{(from~\eqref{eq:embedding(b);wt-lifts-b})}\\
    &\leq 7K\Bignorm{\sum_{s=1}^\infty \beta_s
      v'_{n_{s-1}}}_V+14K\Delta\\
    &\leq 7KD\Bignorm{\sum_{s=1}^\infty \beta_s v''_{t_{n_{s-1}}}}_V
    +14K\Delta\\
    \intertext{(as $(v_i)$ is $D$-left-dominant)}
    &\leq 14CD^2K\Bignorm{\sum_{s=1}^\infty \beta_s\bbar_s}
    +14K\Delta \\
    \intertext{(since $(\bbar_s)$ $2CD$-dominates
      $(v''_{t_{n_{s-1}}})$)}
    &\leq \text{\makebox[0pt][l]{$14CD^2K\Bignorm{\sum_{i=n_0}^\infty
	  \alpha_i\xb_i}+14CD^2K\Delta+14K\Delta$\ .}}
  \end{align*}
  Finally, since $(\xb_i)$ is a basic sequence with projection
  constant at most $2K$, it follows
  from~\eqref{eq:embedding(b);sum-alpha-xb} that
  \[
  \Bignorm{\sum_{i=n_0}^\infty \alpha_i\xb_i}_Z \leq 2K
  \Bignorm{\sum_{i=1}^\infty \alpha_i\xb_i}_Z\leq 2K(K+3)\ .
  \]
  This provides an upper bound of $116CD^2K^3$ for the third term
  of~\eqref{eq:embedding(b);new-quotient3}, which
  leads~\eqref{eq:embedding(b);new-quotient2} with $L\keq 126CD^2K^3$,
  as claimed. This completes the proof of part~(b) of
  Theorem~\ref{thm:embedding}.
\end{proof}

\section{Universal constructions and applications}
\label{section:universal}

Let $V$ and $U$ be reflexive spaces with normalized,
$1$-unconditional, block-stable bases $(v_i)$ and $(u_i)$,
respectively, such that $(v_i)$ is left-dominant, $(u_i)$ is
right-dominant and $(v_i)$ is dominated by $(u_i)$. For each
$C\kin[1,\infty)$ let $\cA_{V,U}(C)$ denote the class of all
separable, infinite-dimensional, reflexive Banach spaces that satisfy
subsequential $C$-$(V,U)$-tree estimates. We also let
\[
\cA_{V,U}=\bigcup_{C\in[1,\infty)} \cA_{V,U}(C)\ ,
\]
which is the class of all separable, infinite-dimensional, reflexive
Banach spaces that satisfy subsequential $(V,U)$-tree estimates.

\begin{thm}
  \label{thm:universal}
  The class $\cA_{V,U}$ defined above contains an element which is
  universal for the class.

  More precisely, for all $B,D,L,R\kin[1,\infty)$ there exists a
  constant $\Cb\keq\Cb(B,D)\kin[1,\infty)$ and for all
  $C\kin[1,\infty)$ there is a constant $K(C)\keq
  K_{B,D,L,R}(C)\kin[1,\infty)$ such that if $(v_i)$ is
  $B$-block-stable and $L$-left-dominant, if $(u_i)$ is
  $B$-block-stable and $R$-right-dominant, and if $(v_i)$ is
  $D$-dominated by $(u_i)$, then there exists $Z\kin\cA_{V,U}$ such
  that for all $C\kin[1,\infty)$ every $X\kin\cA_{V,U}(C)$
  $K(C)$-embeds into $Z$, and moreover $Z$ has a bimonotone FDD
  satisfying subsequential $\Cb$-$(V,U)$ estimates in $Z$.
\end{thm}

\begin{proof}
  By a result of Schechtman~\cite{S} there exists a space $W$ with a
  bimonotone FDD $E\keq(E_i)$ with the property that any bimonotone
  FDD is naturally almost isometric to a subsequence
  $\overline{\bigoplus_{i=1}^\infty E_{k_i}}$ which is
  $1$-complemented in $W$. More precisely, given a Banach space $X$
  with a bimonotone FDD $(F_i)$ and given $\ve\kge 0$, there is a
  subsequence $(E_{k_i})$ of $(E_i)$ and a $(1\kplus\ve)$-embedding
  $T\colon X\to W$ such that $T(F_i)\keq E_{k_i}$ for all $i\kin\bn$,
  and $\sum_{i=1}^\infty P^E_{k_i}$ is a norm-$1$ projection of $W$
  onto $\overline{\bigoplus_{i=1}^\infty E_{k_i}}$.

  We shall now modify the norm on $W$ in two stages. We first consider
  the space $\big(\Ws\big)^{U^*}\big(E^*\big)$. By
  Corollary~\ref{cor:bdd-complete-ZV} the sequence $(E^*_i)$ is a
  boundedly complete (and bimonotone) FDD for this space. It follows
  that $(E_i)$ is a bimonotone, shrinking FDD for a space $Y$ with
  $Y^*\keq\big(\Ws\big)^{U^*}\big(E^*\big)$. By
  Lemma~\ref{lem:ZV_has-V-lower-est} and
  Proposition~\ref{prop:fdd-est-duality} $(E_i)$ satisfies
  subsequential $2B$-$U$-upper estimates in $Y$.

  We now let $Z\keq Y^V(E)$. By Corollary~\ref{cor:refl-ZV} $Z$ is
  reflexive, by Lemma~\ref{lem:ZV_has-V-lower-est} $(E_i)$ satisfies
  subsequential $2B$-$V$-lower estimates in $Z$, and by
  Lemma~\ref{lem:ZV-keeps-U-upper-est} $(E_i)$ also satisfies
  subsequential $(3BD\kplus 2B^2D)$-$U$-upper estimates in $Z$. Thus
  $(E_i)$ is a bimonotone FDD satisfying subsequential
  $\Cb$-$(V,U)$-estimates in $Z$, where $\Cb\keq 3BD\kplus 2B^2D$. It
  remains to show that $Z$ is universal for $\cA_{V,U}$.

  Let $C\kin[1,\infty)$ and let $X\kin\cA_{V,U}(C)$. By
  Theorem~\ref{thm:lower-and-upper-embedding'} there exist constants
  $K_1\keq K_1(B,D,R)$ and $K_2\keq K_2(C,L,R)$ in $[1,\infty)$ such
  that $X$ $K_2$-embeds into a reflexive space $\Xt$ which has a
  bimonotone FDD $(F_i)$ satisfying subsequential $K_1$-$(V,U)$
  estimates in $\Xt$. Now we can find a subsequence $(E_{k_i})$ of
  $(E_i)$ and a $2$-embedding $T\colon \Xt\to W$ such that $T(F_i)\keq
  E_{k_i}$ for all $i\kin\bn$ and $\sum_i P^E_{k_i}$ is a norm-$1$
  projection of $W$ onto $\overline{\bigoplus_iE_{k_i}}$. It follows
  in particular that $(E_{k_i})$ satisfies subsequential
  $2K_1$-$(V,U)$ estimates in $W$, \ie if $(w_i)$ is a normalized
  block sequence of $(E_{k_n})$ in $W$ with $\min\supp_E(w_i)\keq
  k_{m_i}$ for all $i\kin\bn$, then  $(w_i)$ $2K_1$-dominates
  $(v_{m_i})$ and is $2K_1$-dominated by $(u_{m_i})$. Hence by
  Proposition~\ref{prop:fdd-est-duality} $(E^*_{k_i})$ satisfies
  subsequential $2K_1$-$(U^*,V^*)$ estimates in $\Ws$. (Note that the
  dual of the subspace $\overline{\bigoplus_iE_{k_i}}$ of $W$ is
  naturally isometrically isomorphic to the subspace
  $\overline{\bigoplus_iE^*_{k_i}}$ of $\Ws$.) We shall now use this
  to show that the norms $\norm{\cdot}_W,\ \norm{\cdot}_Y$ and
  $\norm{\cdot}_Z$ are all equivalent when restricted to
  $\coo(\oplus_i E_{k_i})$, which implies that $\Xt$ and hence also
  $X$ embed into $Z$.

  Fix $w^*\kin\coo(\oplus_iE^*_{k_i})$. Clearly we have
  $\norm{w^*}_{\Ws}\kleq \norm{w^*}_{Y^*}$. Choose $1\kleq m_0\kle
  m_1\kle\dots$ in $\bn$ such that
  \[
  \norm{w^*}_{Y^*}=\Bignorm{\sum_{i=1}^\infty
  \norm{P^{E^*}_{[m_{i-1},m_i)} w^*}_{\Ws}\kcdot u^*_{m_{i-1}}}_{U^*}\
  .
  \]
  We may assume that $m_0\keq 1$ and $P^{E^*}_{[m_{i-1},m_i)} w^*\kneq
  0$ for all $i\kin\bn$. Then there exist $j_1\kle j_2\kle\dots$ in
  $\bn$ such that $k_{j_i}\keq \min\supp_{E^*}P^{E^*}_{[m_{i-1},m_i)}
  w^*$ for all $i\kin\bn$. Since $(u^*_i)$ is $B$-block-stable and
  $R$-left-dominant, and since $(E^*_{k_i})$ satisfies subsequential
  $2K_1$-$U^*$-lower estimates in $\Ws$, we have
  \begin{align}
    \label{eq:universal;Y*<W*}
    \norm{w^*}_{Y^*} &\leq B\Bignorm{\sum_{i=1}^\infty
      \norm{P^{E^*}_{[m_{i-1},m_i)} w^*}_{\Ws}\kcdot
      u^*_{k_{j_i}}}_{U^*}\\
    &\leq BR\Bignorm{\sum_{i=1}^\infty
      \norm{P^{E^*}_{[m_{i-1},m_i)} w^*}_{\Ws}\kcdot
      u^*_{j_i}}_{U^*}\notag \\
    &\leq  2BRK_1 \norm{w^*}_{\Ws}\ .\notag
  \end{align}
  This shows that $\norm{\cdot}_{\Ws}$ and $\norm{\cdot}_{Y^*}$ are
  equivalent on  $\coo(\oplus_iE^*_{k_i})$. It is easy to verify that
  $\sum_i P^{E^*}_{k_i}$, which defines a norm-$1$ projection of $\Ws$
  onto $\overline{\bigoplus_iE^*_{k_i}}$, is also a norm-$1$
  projection of $Y^*$ onto $\overline{\bigoplus_iE^*_{k_i}}$. It
  follows that
  $\frac1{2BRK_1}\norm{w}_W\kleq\norm{w}_Y\kleq\norm{w}_W$ for all
  $w\kin\coo(\oplus_iE_{k_i})$. 

  A very similar argument shows that $\norm{y}_Y\kleq\norm{y}_Z \kleq
  2BLK_1\norm{y}_Y$ for all $y\kin\coo(\oplus_iE_{k_i})$. Indeed, the
  first inequality is clear from the definition of $\norm{\cdot}_Z$,
  whereas the second one is obtained by a computation similar to the
  one in~\eqref{eq:universal;Y*<W*}.

  We have thus shown that the $2$-embedding $T\colon\Xt\to W$ becomes
  a $8B^2LRK_1^2$-embedding viewed as a map $\Xt\to Z$. Hence $X$
  $K$-embeds into $Z$, where $K\keq 8B^2LRK_1^2K_2$.
\end{proof}

We conclude this paper with two applications of our embedding
theorems. The first one is the observation that our results here give
an alternative proof to the main theorem in~\cite{OS2}.

\begin{thm}[\cite{OS2}]
  Let $X$ be a separable, reflexive Banach space and let $1\kleq q\kleq
  p\kleq\infty$. The following are equivalent.
  \begin{mylist}{(a)}
  \item[(a)]
    $X$ satisfies $(p,q)$ tree estimates.
  \item[(b)]
    $X$ is isomorphic to a subspace of a reflexive space $Z$ having an
    FDD which satisfies $(p,q)$ estimates.
  \item[(c)]
    $X$ is isomorphic to a quotient of a reflexive space $Z$ having an
    FDD which satisfies $(p,q)$ estimates.
  \end{mylist}
\end{thm}
Here an FDD $(E_n)$ of a Banach space $Z$ is said to satisfy $(p,q)$
estimates if there is a constant $C\kge 0$ such that for every block
sequence $(x_i)$ of $(E_n)$ we have
\[
C^{-1} \bigg(\sum\norm{x_i}^p\bigg)^{1/p} \leq \Bignorm{\sum
  x_i} \leq C \bigg(\sum\norm{x_i}^q\bigg)^{1/q}\ ,
\]
and a Banach space $X$ is said to satisfy $(p,q)$ estimates if every
normalized, weakly null tree $(x_\alpha)_{\alpha\in T_\infty}$ in $X$
has a branch that dominates the unit vector basis of $\ell_p$ and that
is dominated by the unit vector basis of $\ell_q$. The family
$(x_\alpha)_{\alpha\in T_\infty}$ in $X$ is called a normalized,
weakly null tree if for all $\alpha\kin T_\infty\cup\{\emptyset\}$ the
sequence $\big(x_{(\alpha,n)}\big)$ is normalized and weakly null, and
a branch of $(x_\alpha)_{\alpha\in T_\infty}$ is a sequence
$\big(x_{(n_1,n_2,\dots,n_i)}\big)$, where $n_1\kle n_2\kle\dots$.

The second, and main, application concerns the existence of universal
spaces for the classes $C_\alpha$ defined in the Introduction. Recall
that for each countable ordinal $\alpha$ the class $C_\alpha$
consists of all separable, reflexive spaces $X$ such that both $X$ and
its dual $X^*$ have Szlenk index at most $\alpha$. Szlenk introduced
his index to show that there is no separable, reflexive space that
contains isomorphic copies of every separable, reflexive
space~\cite{Sz}.

The Szlenk index $\sz(\cdot)$ has the following properties~\cite{Sz}:
for a separable space $X$, $\sz(X)\kle\omega_1$ if and only if $X^*$
is separable (so $\bigcup_{\alpha<\omega_1} C_\alpha$ is the class of
all separable, reflexive spaces); if $Y$ embeds into $X$, then
$\sz(Y)\kleq\sz(X)$; for all $\alpha\kle\omega_1$ there exists a
separable, reflexive space $X$ such that $\sz(X)\kge\alpha$. From
these properties it follows immediately that if a separable space $Z$
contains isomorphic copies of every separable, reflexive space, then
$Z^*$ is not separable, and so $Z$ cannot be reflexive. (Later
J.~Bourgain showed that such a space $Z$ must contain
$C[0,1]$, and hence all separable Banach spaces~\cite{B}.)

It seems natural to ask if there is, for each countable ordinal
$\alpha$, a separable, reflexive space that is universal for
$C_\alpha$. This question was indeed raised by Pe\l czy\'nski
motivated by the results of~\cite{OS2}, which imply an affirmative
answer for $\alpha\keq\omega$. In~\cite{OSZ2} we show that Pe\l
czy\'nski's question has an affirmative answer for all
$\alpha\kle\omega_1$.

\begin{thm}
  For each countable ordinal $\alpha$ there is a separable, reflexive
  space $Z$ which is universal for the class $C_\alpha$.
\end{thm}

This is a simplified version of our result which also includes
estimates on embedding constants and determines the class $C_\beta$ in
which the universal space $Z$ lives. The proof, which is given
in~\cite{OSZ2}, splits into two parts. We first prove that if $X\kin
C_\alpha$, then there exists $\gamma\kle\alpha$ such that $X$
satisfies subsequential $((T_{\gamma,\frac12})^*,T_{\gamma,\frac12})$
tree estimates, where $T_{\gamma,\frac12}$ is the Tsirelson space of
order $\gamma$. The ingredient for the second part of the proof is
Theorem~\ref{thm:universal} from this paper. We fix a sequence
$(\alpha_n)$ of ordinals with $\alpha\keq\sup_n(\alpha_n\kplus 1)$,
and for each $n\kin\bn$ we let $Z_n$ be a separable, reflexive space
which is universal for the class
$\cA_{((T_{\alpha_n,\frac12})^*,T_{\alpha_n,\frac12})}$. The
$\ell_2$-direct sum $Z$ of the sequence $(Z_n)$ is then the required
universal space for $C_\alpha$.

\end{document}